 \documentclass[11pt]{amsart}
\usepackage{amsmath}
\usepackage{amssymb}
\usepackage{graphicx}
\usepackage{stmaryrd}
\usepackage{tikz}
\usepackage{psfrag}
\usepackage{bm}

\newtheorem{theorem}{Theorem}[section]

\newtheorem{lemma}[theorem]{Lemma}

\theoremstyle{definition}
\newtheorem{definition}{Definition}[section]

\newtheorem{remark}{Remark}[section]

\numberwithin{equation}{section}

\newcommand{\B}{\mathbb B}
\newcommand{\R}{\mathbb R}

\newcommand{\U}{\mathcal U}

\newcommand{\A}{\mathbb A}

\newcommand{\T}{\mathcal T}

\newcommand{\dist}{\operatorname{dist}}

\newcommand{\dv}{\operatorname{div}}
\renewcommand{\mod}{\operatorname{mod}}

\newcommand{\rad}{\operatorname{rad}}

\title[Divergence-form equations with nowhere $C^1$ Lipschitz  solutions]{Divergence-form equations admitting nowhere $C^1$ Lipschitz weak solutions}

\date{\today}

\author{Menglan Liao}
\address{School of Mathematics, Hohai University, Nanjing, Jiangsu Province, China}
\email{liaoml@hhu.edu.cn}

 \author{Baisheng Yan}
 \address{Department of Mathematics\\ Michigan State University\\ East Lansing, MI 48824, USA}
   \email{yanb@msu.edu}

\keywords{Divergence-form equations,  differential relations,  convex integration, $\T_N$-configurations, Condition $O_N$, Lipschitz weak solutions that are nowhere $C^1$}


\begin{document}

\begin{abstract}
We study a class of partial differential equations in divergence form that admit highly irregular Lipschitz weak solutions. By reformulating these divergence-form equations as a first-order partial differential relation and adapting the convex integration scheme recently developed in \cite{GKY26} for irregular diffusion equations, we show that the same structural Condition~$O_N$ introduced there also ensures the existence of Lipschitz weak solutions that are nowhere $C^1$ for the corresponding time-independent equations in bounded domains, under suitable boundary data. In particular, for the smooth strongly polyconvex functions on $\mathbb{R}^{2\times n}$ constructed  in that paper  for all $n \ge 2$, the associated Euler--Lagrange equations admit Lipschitz weak solutions that are nowhere $C^1$ and satisfy zero boundary conditions in any bounded domain of $\mathbb{R}^n$.  
Our approach relies on new building blocks constructed from the same wave cone and $\mathcal{T}_N$-configurations employed in the analysis of diffusion equations.
\end{abstract}

\maketitle

%%%%%%%%%%%%%%%%%%%%%%%%%%%%%%%%%%%%%%%%%%%%%%%%%%%%%%%%%%%%%%%%%%%%%%%%%%%%%%%%

\section{Introduction and Main Results}

We study a class of partial differential equations in divergence form:
\begin{equation}\label{DE}
\operatorname{div} \sigma(Du(x)) = 0 \quad \text{in } \Omega,
\end{equation}
where $\Omega \subset \mathbb{R}^n$ is a bounded domain and 
$u \colon \Omega \to \mathbb{R}^m$, $u(x) = (u^1,\dots,u^m)$, is the unknown function with gradient matrix
\[
Du(x) = \left( \frac{\partial u^i}{\partial x_j} \right) \in \mathbb{R}^{m \times n}.
\]
Here $m,n \ge 1$, and $\sigma \colon \mathbb{R}^{m \times n} \to \mathbb{R}^{m \times n}$ is a given continuous function. 
We denote by $\mathbb{R}^{m \times n}$ the Euclidean space of real $m \times n$ matrices with the inner product
$
\langle A, B \rangle = \operatorname{tr}(A^{T}B).
$

By a (very) \emph{weak solution} of equation \eqref{DE} (in the sense of distributions), we mean a function $u\in  W^{1,1}_{\mathrm{loc}}(\Omega;\R^m)$ such  that $\sigma(Du) \in L^1_{\mathrm{loc}}(\Omega;\R^{m\times n})$  and  
\begin{equation}\label{weak-sol}
 \int_{\Omega}  \langle  \sigma(D u(x)), D\varphi(x) \rangle\, dx =0\quad  \forall\, \varphi\in  C^\infty_0(\Omega;\R^m).
\end{equation}
This condition is equivalent to the row-wise formulation
\[
 \int_{\Omega}  \Big ( \sum_{j=1}^n  \sigma_{ij}(D u(x))  \partial_{x_j}\phi(x) \Big ) dx =0\quad  \forall\, \phi\in  C^\infty_0(\Omega),
 \quad i=1,\dots, m. 
 \] 

 When $m,n \ge 2$, equation \eqref{DE} is a system of $m$ partial differential equations in  $\Omega\subset \R^n.$ 
In this paper, we are primarily concerned with equations \eqref{DE} that admit {\em Lipschitz} weak solutions which are nowhere $C^1$ in $\Omega$. 
Such equations are referred to as \emph{irregular}.

Structural assumptions on $\sigma$ are central to the study of \eqref{DE},
particularly for systems. For example, $\sigma$ is called 
\emph{strongly rank-one monotone} if
\begin{equation}\label{r-mono}
\langle \sigma(A + p \otimes a) - \sigma(A), \, p \otimes a \rangle
\ge \nu |p|^2 |a|^2
\end{equation}
for all $A \in \mathbb{R}^{m \times n}$, $p \in \mathbb{R}^m$, and
$a \in \mathbb{R}^n$, with some $\nu > 0$.
In this case, equation \eqref{DE} is usually said to be  \emph{strongly elliptic} in the
Legendre--Hadamard sense. Moreover, $\sigma$ is called \emph{strongly quasimonotone} if
\begin{equation}\label{q-mono}
\int_\Omega \langle \sigma(A + D\varphi), \, D\varphi \rangle \, dx
\ge \nu \int_\Omega |D\varphi|^2 \, dx
\end{equation}
for all $A \in \mathbb{R}^{m \times n}$ and
$\varphi \in C_0^\infty(\Omega;\mathbb{R}^m)$; see
\cite{CZ92,Fu87,Ha95,KY25, La96,Zh86}.

It is standard in the calculus of variations that \eqref{q-mono}
implies \eqref{r-mono}. Clearly, \eqref{q-mono} holds if $\sigma$ is \emph{strongly monotone}, that is, 
\begin{equation}\label{f-mono}
\langle \sigma(A + B) - \sigma(A), \, B \rangle
\ge \nu |B|^2
\quad \forall\, A,B \in \mathbb{R}^{m \times n}.
\end{equation}
When $m = 1$ or $n = 1$, the conditions \eqref{r-mono}--\eqref{f-mono}
are equivalent.

For smooth functions $\sigma$ satisfying \eqref{q-mono}, the results of
\cite{DG00,Fu87,Ha95}   imply that every Lipschitz
weak solution $u$ of equation \eqref{DE} enjoys \emph{partial 
$C^{1,\alpha}$ regularity}. More precisely,
\[
Du \in C^\alpha_{\mathrm{loc}}(U; \mathbb{R}^{m \times n})
\]
for some $\alpha \in (0,1)$ and some open set $U \subset \Omega$ of full
measure. In this sense, such equations are  regarded as \emph{regular}.  

When $\sigma = DF$ for a smooth function $F \colon \mathbb{R}^{m\times n} \to \mathbb{R}$, equation~\eqref{DE} arises as the Euler--Lagrange equation associated with the energy functional
\begin{equation}\label{energy}
I(u) = \int_\Omega F(Du(x)) \, dx.
\end{equation}
See, for example, \cite{AF84,Ba77,Da08,Ev86,FH85, Gi83,Mo52} for general treatments of such functionals. In particular, the results of 
\cite{Ev86,FH85} established partial $C^{1,\alpha}$ regularity for minimizers of $I(u)$ under a strong quasiconvexity assumption on $F$, which is weaker than condition~\eqref{q-mono}. See also \cite{DG00,Fu87,Ha95, Mi06, Zh92} for related regularity results.

In the celebrated works \cite{MSv03,Sz04}, Lipschitz weak solutions that are nowhere $C^1$ were constructed for equation~\eqref{DE} in the case $m=n=2$ by adapting the convex integration method \cite{Gr86} and employing the so-called \emph{$T_N$-configurations}  \cite{Ta93} for $N=4,5$. In these works, $\sigma = DF$ arises from suitably chosen quasiconvex or polyconvex functions $F$ on $\mathbb{R}^{2 \times 2}$, so that the resulting equations form irregular elliptic systems.

More recently, highly irregular very weak solutions were constructed in \cite{CT25,Jo24} for the scalar $p$-Laplace equation and two-dimensional elliptic Euler--Lagrange equations, using a different convex integration scheme based on the so-called \emph{staircase laminates}.

Most recently, building on the convex integration framework developed in \cite{KY15,KY17,KY18,Ya20,Ya22,Zh06}, the irregularity of the diffusion equation 
\begin{equation}\label{PE}
\partial_t u = \operatorname{div} \sigma(Du)
\quad \text{in } \Omega_T := \Omega \times (0,T), \quad T>0,
\end{equation}
was investigated in \cite{GKY26} under a structural assumption on the function $\sigma$, termed \emph{Condition $O_N$}, which is formulated in terms of certain geometric structures called \emph{$\mathcal{T}_N$-configurations}; see Definitions~\ref{T-N} and \ref{O-N} below. Under this assumption, the associated initial-boundary value problems for \eqref{PE} admit infinitely many Lipschitz weak solutions that are nowhere $C^1$ in $\Omega_T$, even for certain smooth initial and boundary data. This stands in sharp contrast to the partial $C^{1,\alpha}$ regularity result established in \cite{BDM13} for equation \eqref{PE} under condition \eqref{q-mono}.

In this paper, we show that the same Condition $O_N$ also guarantees the existence of Lipschitz weak solutions to equation \eqref{DE} that are nowhere $C^1$ in $\Omega$, for certain smooth boundary data. More precisely, we prove the following main result.

\begin{theorem}\label{mainthm}
Let $m,n \ge 1$, and let $\sigma\colon \R^{m\times n}\to \R^{m\times n}$ be locally Lipschitz and satisfy Condition~$O_N$ for some $N\ge 2$. Let $\Sigma(1)$ denote the open set defined in Definition~\ref{O-N}.

Assume that $\bar u \in C^1(\bar \Omega;\R^m)$ and $\bar V \in C(\bar\Omega; \R^{m\times n})$ satisfy
\[
\dv \bar V = 0 \quad \text{in } \Omega, 
\quad 
(D\bar u, \bar V) \in \Sigma(1) \quad \text{on } \bar \Omega.
\]
Here and throughout, the divergence-free condition is understood in the sense of distributions.

Then, for any $\delta \in (0,1)$, the Dirichlet problem
\begin{equation}\label{DP}
\begin{cases}
\operatorname{div} \,\sigma(Du) = 0 & \text{in } \Omega,\\[2mm]
u = \bar u & \text{on } \partial \Omega,
\end{cases}
\end{equation}
admits a Lipschitz weak solution $u$ that is nowhere $C^1$ in $\Omega$ and satisfies
\begin{equation}\label{c-0}
\|u - \bar u\|_{L^\infty(\Omega)} < \delta.
\end{equation}
\end{theorem}

We note that many non-monotone functions $\sigma$ representing forward-backward diffusion, studied in \cite{KY15,KY17,KY18,Zh06} in the scalar case ($m=1$), satisfy Condition $O_2$. The functions $\sigma = DF$ constructed in \cite{MSv03,Sz04} for $m = n = 2$, mentioned above, can be shown to satisfy Conditions $O_4$ and $O_5$, respectively. Moreover, functions of the form $\sigma = DF$, corresponding to certain polyconvex functions $F$ on $\R^{2\times n}$ constructed in \cite{GKY26} for all $n \ge 2$, also satisfy Condition $O_5$.
 Therefore, Theorem~\ref{mainthm} applies in all these cases. 

The following result follows from Theorem~\ref{mainthm} using the strongly polyconvex functions on $\mathbb{R}^{2 \times n}$ constructed in \cite[Theorem~2.2]{GKY26}. See also the related result in \cite[Corollary~2.3]{GKY26}.

\begin{theorem}\label{mainthm2}
Let $m,n \ge 2$. Then there exist (strongly polyconvex) functions  $F$ on $\R^{m\times n}$ of the form
\begin{equation}\label{poly-00}
F(A) = \frac{\nu}{2} |A|^2 + G\big(A_1, \delta(A_1)\big)
\quad \forall\, 
A = 
\begin{pmatrix}
A_1 \\
A_2
\end{pmatrix}
\in \mathbb{R}^{m \times n},
\end{equation}
where $\nu > 0$ is a constant, 
$A_1= (a_{ij})  \in \mathbb{R}^{2 \times n}$, $\delta(A_1) = a_{11} a_{22} - a_{12} a_{21},$ and $G \in C^\infty(\mathbb{R}^{2 \times n} \times \mathbb{R})$ is  convex, such that for some constants $C_k > 0$,
\begin{equation}\label{poly-0}
|D^k F(A)| \le C_k \big( |A| + 1 \big)
\quad \forall\, A \in \mathbb{R}^{m \times n}, \;\; k = 1,2,\dots,
\end{equation}
and such that the Dirichlet problem
\[
\begin{cases}
\operatorname{div} DF(Du) = 0 & \text{in } \Omega, \\
u = 0 & \text{on } \partial \Omega,
\end{cases}
\]
admits infinitely many Lipschitz weak solutions that are nowhere $C^1$ in~$\Omega$.
\end{theorem}

\begin{remark}
\leavevmode

(i) Functions  $F$ of the form   \eqref{poly-00} are \emph{strongly quasiconvex} in the sense that
\begin{equation}\label{qc}
\int_\Omega \big[ F(A + D\varphi) - F(A) \big] \, dx 
\ge 
\frac{\nu}{2} \int_\Omega |D\varphi|^2 \, dx 
\quad 
\forall\, A \in \mathbb{R}^{m \times n},\; 
\varphi \in C_0^\infty(\Omega; \mathbb{R}^m).
\end{equation}
 It is clear that \eqref{qc} holds if $\sigma = DF$ satisfies \eqref{q-mono}. 
However, Theorem~\ref{mainthm2}, in view of the regularity results of \cite{Ev86,FH85}, shows that condition \eqref{qc} is strictly weaker than \eqref{q-mono} with $\sigma = DF$, even for polyconvex functions $F$. Explicit related examples can also be found in \cite{CZ92,Ha95,KY25}.

(ii)  As shown in \cite{Zh92}, condition \eqref{qc} together with the assumption in \eqref{poly-0} for $k=3$  implies the following: for any Lipschitz weak solution $u$ of \eqref{DE}, if $Du$ is continuous at a point $x_0 \in \Omega$, then $u$ is a local minimizer of the energy $I(u)$ at $x_0$, in the sense that there exists a ball $B_r(x_0) \subset\subset \Omega$ such that
\[
\int_{B_r(x_0)} F(Du)\, dx 
\le 
\int_{B_r(x_0)} F(Du + D\psi)\, dx \quad\forall\, \psi \in C_0^\infty(B_r(x_0); \mathbb{R}^m).
\]
Therefore, the results of \cite{Ev86,FH85} show that any solution $u$ as in Theorem~\ref{mainthm2} cannot be a local minimizer of the energy $I(u)$ at any point in $\Omega$. This provides a negative answer to the question posed in \cite{Zh92}.
\end{remark}

\begin{proof}[Proof of Theorem \ref{mainthm2}] 
Let 
\[
F_1(A_1)=\frac{\nu}{2}|A_1|^2+G(A_1,\delta(A_1)) \quad \forall\, A_1=(a_{ij})\in \R^{2\times n}
\]
 be as  constructed   in \cite[Theorem 2.2]{GKY26}; in particular, $G \in C^\infty(\mathbb{R}^{2 \times n} \times \mathbb{R})$ is  convex and  the function $\sigma_1=DF_1\colon \R^{2\times n}\to\R^{2\times n}$ satisfies Condition $O_5.$ Since $(0,0)\in\Sigma(1)$ for this $\sigma_1$ (see \cite[Remark 2.1]{GKY26}), by Theorem \ref{mainthm},   the  Dirichlet  problem
\begin{equation}\label{dp}
\begin{cases}
\operatorname{div} DF_1(Du_1) = 0 & \text{in } \Omega, \\
u_1 = 0 & \text{on } \partial \Omega,
\end{cases}
\end{equation}
admits  infinitely many  Lipschitz weak solutions $u_1\colon \Omega\to \R^2$ that are nowhere $C^1$  in $\Omega.$
This proves the result when $m = 2$.

Now assume that $m\ge 3.$ Define $F\colon \R^{m\times n}\to\R$ by
\[
F(A)=\frac{\nu}{2}|A_2|^2 +F_1(A_1) \qquad \forall\,A=\begin{pmatrix}A_1\\A_2\end{pmatrix},\;\; A_1\in \R^{2\times n}. 
\]
Then $F$ satisfies \eqref{poly-0}, and moreover
\[
F(A)=\frac{\nu}{2}|A|^2 +G(A_1,\delta(A_1)), \quad DF(A)= \begin{pmatrix} DF_1(A_1)\\ \nu A_2 \end{pmatrix}.
\]

 Let $u(x)=(u_1(x), \,0) \in \R^2\times \R^{m-2}=\R^m,$ where $u_1\colon \Omega\to \R^2$ is a nowhere $C^1$ Lipschitz weak solution of   \eqref{dp}. Then $u$   is also nowhere $C^1$  in $\Omega.$ Moreover,
\[
 Du=\begin{pmatrix} Du_1\\O\end{pmatrix}, \quad DF(Du)=\begin{pmatrix} DF_1(Du_1)\\O\end{pmatrix}.
\]
Therefore,  $u$ is a Lipschitz weak solution of
\[
\begin{cases}
\operatorname{div} DF(Du) = 0 & \text{in } \Omega, \\
u = 0 & \text{on } \partial \Omega,
\end{cases}
\]
which completes the proof.
 \end{proof}

The remainder of the paper is devoted to the proof of Theorem~\ref{mainthm}.  We adapt the convex integration framework by reformulating equation~\eqref{DE} as a first-order partial differential relation:
\begin{equation}\label{pdr1}
\operatorname{div} V = 0 \;  \text{ in } \Omega, \quad (Du, V) \in \mathcal{K} \;  \text{ a.e. in } \Omega,
\end{equation}
for the unknown function \((u, V) \colon \Omega \to \mathbb{R}^m \times \mathbb{R}^{m \times n}\), where
\[
\mathcal{K} = \{(A, \sigma(A)) : A \in \mathbb{R}^{m \times n}\}
\] 
denotes the graph of the function $\sigma$.

Although this differential relation  falls within the general framework of partial differential relations studied in \cite{DM97,DM99} using Baire category methods, and in \cite{MSv99,MSy01} via convex integration, it requires special treatment due to the new Condition $O_N$ and the goal of constructing  nowhere $C^1$ solutions. Here, we employ the same sequence of uniformly bounded open sets,
\[
\{\mathcal{U}_\nu\}_{\nu=1}^\infty \subset \mathbb{R}^{m \times n} \times \mathbb{R}^{m \times n},
\]
as constructed in \cite{GKY26}, but we build a new corresponding sequence 
\[
\{(u_\nu, V_\nu)\}_{\nu=1}^\infty \subset (\bar u, \bar V) + C^\infty_0(\Omega; \mathbb{R}^m \times \mathbb{R}^{m \times n}) 
\]
such that
\begin{equation}\label{pdr2}
\begin{cases} 
\dv V_\nu = 0  \;  \text{ in } \Omega, \quad (Du_\nu, V_\nu) \in \mathcal{U}_\nu \;\text{ on } \bar \Omega, \\
\displaystyle \lim_{\nu \to \infty} \|\sigma(Du_\nu) - V_\nu\|_{L^1(\Omega)} = 0,\\
\{u_\nu\}_{\nu=1}^\infty \text{ is Cauchy in } W^{1,1}(\Omega; \mathbb{R}^m).
\end{cases}
\end{equation}

The limit $u$ of $\{u_\nu\}_{\nu=1}^\infty$ in $W^{1,1}(\Omega; \mathbb{R}^m)$ then provides  a Lipschitz weak solution to the Dirichlet problem  \eqref{DP}. 
Moreover, the construction also ensures that the essential oscillation of $Du$ at every point in $\Omega$ is uniformly bounded below by a positive constant, so  that $u$ is nowhere $C^1$ in $\Omega$.

The key steps in our convex integration strategy rely on  the new initial  building blocks  (see Lemma~\ref{lem-0}) derived from the same wave cone
\begin{equation}\label{set-G}
\Gamma = \big\{ (p \otimes a, B) : p \in \mathbb{R}^m, \; a \in \mathbb{R}^n \setminus \{0\}, \; B \in \mathbb{R}^{m \times n}, \; Ba = 0 \big\},
\end{equation}
as introduced  in \cite{GKY26}.  By combining these blocks with $\T_N$-configurations under Condition~$O_N$, we implement the convex integration scheme of \cite{GKY26} to establish the main stage theorem (Theorem~\ref{thm1}), which in turn realizes the construction of \eqref{pdr2} and ultimately completes the proof of Theorem~\ref{mainthm}.

%%%%%%%%%%%%%%%%%%%%%%%%%%%%%%%%%%%%%%%%%%%%%%%%%%%%%%%%%%%%%%

\section{Convex Integration,  $\T_N$-configurations and  Condition $O_N$}\label{s-3}

\subsection{Notation} We introduce some notation  used throughout the paper.

  Let $d, q\ge1$ be  integers.  A {\em cube}   in $\R^{d}$   always refers to an open cube whose sides are parallel to the coordinate axes.  For $y\in\R^d$ and $l>0$, we denote by $Q_{y,l}$   the  cube centered at  $y$ with side length $2l,$ and write  $l=\rad(Q_{y,l}).$

Given  {$S\subset\R^d$}  and  $\epsilon>0$,  we denote by $S_\epsilon$   the $\epsilon$-neighborhood of $S;$ i.e.,
\[
S_\epsilon=\{\xi \in \R^d : \dist(\xi,S)<\epsilon\}
=\bigcup_{\eta\in S}\{\xi \in \R^d : |\eta-\xi|<\epsilon\}.
\]
Given $\xi,\eta\in \R^q$,  the closed line segment  
connecting $\xi$ and $\eta$ is   denoted  by 
\[
[\xi,\eta]=\{\lambda \xi+(1-\lambda)\eta: \lambda\in [0,1]\}.
\]

The Lebesgue measure of a measurable set $G$ in $ \R^d$ is denoted by $|G|=\mathcal L_d(G).$

 Let  $G\subset \R^d$ be  a bounded open set, and let  $f\colon G\to \R^q$ be  a measurable function.  
For any $x_0\in G$,  the  \emph{essential oscillation}  of $f$ at $x_0$ is defined by
\[
\omega_f(x_0)=\inf\big\{\|f(x)-f(y)\|_{L^\infty(U\times U)} : U\subset G, \;  \mbox{$U$  open,} \;  x_0\in  U \big\}.
\]
We say that $f$ is \emph{essentially continuous} at $x_0$  if $\omega_f(x_0)=0.$
Clearly, continuity at a point  implies essential continuity there, though the converse need not hold.

Finally, as mentioned in Theorem \ref{mainthm}, given a function $V\in L^1_{\mathrm{loc}}(G;\R^{q\times d})$,  
the divergence-free condition
\begin{equation}\label{div-0}
\dv V=0 \; \text{ in $G$}
\end{equation}
is always understood in the sense of distributions; that is,
\[
 \int_{G}  \langle  V(x), D\varphi(x) \rangle\, dx =0\quad  \forall\, \varphi\in  C^\infty_0(G;\R^q).
\]
Of course,  when $V\in C^1_{\mathrm{loc}}(G;\R^{q\times d})$, this condition coincides with the classical pointwise requirement: 
$\dv V(x)=0 $ for all $x\in G$.

\subsection{Convex Integration Lemma} 

Let $m,n\ge 1,$ and let $\Gamma\subset  \R^{m\times n}\times \R^{m\times n}$ be  the closed wave cone defined by \eqref{set-G}.   

The following result, similar to \cite[Lemma 3.1]{GKY26}, constructs new initial building blocks from the wave cone $\Gamma$.

\begin{lemma}\label{lem-0} Let  $\gamma=(p\otimes a, B)\in  \Gamma$ and $0\le\lambda\le 1.$  
Then, for any bounded open  set $G\subset \R^{n}$ and $0<\epsilon<1,$   there exists a  function  $(\varphi,  \Psi) \in C^\infty_0(G;  \R^m\times  \R^{m\times n})$  with the following properties:
\begin{itemize}
\item[(a)]  $\dv\Psi =0$ and $(D\varphi,\Psi)\in [-\lambda\gamma, (1-\lambda)\gamma]_\epsilon$  in $G;$
\item[(b)]  $\|\varphi\|_{L^\infty(G)}  <\epsilon;$
\item[(c)] there exist two disjoint open  sets $G', G''\subset\subset G$ such that
\[
\begin{cases}
|G'| \ge (1-\epsilon) \lambda |G|, & (D\varphi,\Psi)=(1-\lambda) \gamma  \;\;  \text{\rm in $G'$;}  \\
|G''| \ge (1-\epsilon) (1-\lambda) |G|, & (D\varphi,\Psi)=-\lambda \gamma \; \;   \mbox{\rm in $G''$.}
\end{cases}
\]
\end{itemize}
\end{lemma}

\begin{proof}   If $\lambda=0$ or $\lambda=1$, the result follows by taking  $(\varphi,  \Psi)\equiv (0,0)$ and setting  $G'=\emptyset$ or $G''=\emptyset,$ respectively.
Hence, we assume $0<\lambda<1.$

Let $\zeta\in C^\infty_0(G)$ be such that
\[
\mbox{$0\le \zeta\le 1$ in $G,$  \; \;  $\zeta|_{\tilde G}=1$,  \; where $\tilde G\subset\subset G$ is open, and  $|\tilde G|\ge (1- \epsilon)^{1/3}|G|,$ }
\]
and let $I_1$ and $I_2$ be disjoint closed intervals   in $(0,1)$ such that
\[
  |I_1| = (1-\epsilon)^{1/3} \lambda, \quad
 |I_2| = (1-\epsilon)^{1/3} (1-\lambda).
\]

First, select a function $q\in C^\infty_0(0,1)$ satisfying 
\[
\begin{cases}
 -\lambda \le q \le 1-\lambda \;  \text{ in $ (0,1)$},\\
   \{t\in (0,1):q(t)=1-\lambda\}=I_1,\\
  \{t\in (0,1):q(t)=-\lambda\}=I_2,
 \\
     \int_0^1 q(t)dt =0, \end{cases}
\]
 extend $q$ to   $\R$ as a 1-periodic function, and define
  \[
  f(t)=\int_0^t q(s)\,ds\qquad   \forall\, t\in \R.
  \]
 Note that $f\in C^\infty(\R)$ is 1-periodic,  and  $f'(t)=q(t)$ for all $t\in \R.$ 
 
Next, let $\delta\in (0,1)$ be a parameter  to be chosen later. Define function
\[
h(x) = \zeta(x)f\big (\frac{a\cdot x}{\delta}\big ), 
\]
and set
\[
 \varphi(x)= \delta h(x)p,  \quad  \Psi (x)= \frac{\delta}{ |a|^2} \Big [ (a\cdot Dh(x))B-(BDh(x))\otimes a \Big ].
       \]
 It is easily verified that $(\varphi,  \Psi) \in C^\infty_0(G;  \R^m\times  \R^{m\times n})$, and
\[
 \begin{cases}
 \dv\Psi=0,\\
D \varphi= \zeta(x) f' \big (\frac{a\cdot x}{\delta}\big )p\otimes a+\delta f \big (\frac{a\cdot x}{\delta}\big )p\otimes D\zeta(x),\\
 \Psi=\zeta(x) f' \big (\frac{a\cdot x}{\delta}\big )B +\frac{\delta}{ |a|^2} f\big (\frac{a\cdot x}{\delta}\big ) 
   \big [ (a\cdot D\zeta(x))B -  (BD\zeta(x))\otimes a\big ].
 \end{cases}
\]
Hence  
\[
 \varphi =O(\delta), \quad  (D\varphi,\Psi) =f' \zeta \gamma +O(\delta),  
   \]
 where each $O(\delta)$ term satisfies  $|O(\delta)|\le C\delta$ for  some constant $C$ independent of $\delta\in (0,1).$ 
 
 Since $f'\zeta \gamma \in [-\lambda\gamma,(1-\lambda)\gamma]$ in $G$, it follows that, for all sufficiently small $\delta>0,$
\[
 \|\varphi\|_{L^\infty(G)}  <\epsilon, \quad   
 (D\varphi,\Psi) \in [-\lambda\gamma,(1-\lambda)\gamma]_\epsilon \;\;\mbox{ in $G.$} 
 \]

Let $I_k^o$ be the interior of interval $I_k$ for $k=1,2,$ and let $\chi_k=\chi_{\tilde I_k}$  be the characteristic function  of the open set 
\[
\tilde I_k=\bigcup_{j\in \{0,\pm 1,\pm 2,\dots\}}  (j+ I_k^o). 
\] 
Define the open sets
  \[
  \tilde G_k^\delta= \Big \{x\in \tilde G: \,  \frac{\alpha\cdot x}{\delta}\in \tilde I_k \Big\}= \Big \{x\in \tilde G: \, \chi_k ( \frac{\alpha\cdot x}{\delta}) =1 \Big\},  \quad k=1,2.
  \] 
Then
\[
 (D\varphi,\Psi)|_{\tilde G_1^\delta} =f'\gamma  =(1-\lambda)\gamma, \quad
  (D\varphi,\Psi)|_{\tilde G_2^\delta} =f'\gamma =-\lambda \gamma.
  \]
  
Let $\{\alpha_i\}_{i=1}^{n}$, with $\alpha_1= \frac{\alpha}{|\alpha|}$, be an orthonormal basis of $\R^{n}.$ By the change of variables  $ x=\sum_{i=1}^{n} z_i\alpha_i$  and Fubini's  theorem, for each $k=1,2,$ we have, with $\delta'=\delta/|\alpha|,$  
 \[
\begin{split}  |\tilde G_k^\delta|&=\int_{\tilde G}  \chi_k( \frac{\alpha\cdot x}{\delta})\,dx=\int_{\R^{n}} \chi_{\tilde G}(x) \chi_k( \frac{\alpha_1\cdot x}{\delta'})\,dx\\
&  =\int_{\R^{n}} \chi_{\tilde G'}(z) \chi_k( \frac{z_1}{\delta'})\,dz   =\int_{\R} \Big ( \int_{\R^{n-1}}\chi_{G'(z_1)}(z') \chi_k( \frac{z_1}{\delta'})\,dz'\Big) \,dz_1\\
& =\int_{\R} |G'(z_1)|\chi_k(\frac{z_1}{\delta'})\,dz_1,\end{split}
  \]
where the sets $\tilde G'$ and $G'(z_1)$ are defined by
\[
\tilde G'=\{z\in\R^{n} : y=\sum_{i=1}^{n} z_i\alpha_i\in \tilde G\}, \quad  G'(z_1)= \{z' \in\R^{n-1}: (z_1,z')\in  \tilde G'  \}.
\]

Since $\chi_k$ is 1-periodic for each $k=1,2,$   the sequence $\{\chi_k(\ell t)\}_{\ell=1}^\infty$ converges weakly*  in $L^\infty(\R)$ to the  constant  $C_k=\int_0^1 \chi_k(t)\,dt=|I_k|$.Setting $\delta = \delta_\ell = \frac{|\alpha|}{\ell}$ for $\ell = 1,2,\dots$, we obtain, for each $k=1,2$, the limit
\[
|\tilde G_k^{\delta_\ell}| 
= \int_{\R} |G'(z_1)| \, \chi_k(\ell z_1)\, dz_1 \to  |I_k| \int_{\R} |G'(z_1)|\, dz_1
= |I_k| \, |\tilde G'| = |I_k| \, |\tilde G|,
\quad \text{as } \ell \to \infty.
\]

Since $|I_1|=(1-\epsilon)^{1/3} \lambda, \, |I_2|=(1-\epsilon)^{1/3} (1-\lambda)$, and $|\tilde G|\ge (1-\epsilon)^{1/3}|G|,$ we have
 \[
  |I_1||\tilde G| >(1-\epsilon)  \lambda |G|, \quad
 |I_2| |\tilde G| >(1-\epsilon) (1-\lambda)|G|.
 \]
 Consequently, requirements (a)–(c) are satisfied by choosing
$\delta=\delta_\ell>0$ with $\ell$ sufficiently large and taking the open sets
\[
G'=\tilde G_1^{\delta}, \quad
G''=\tilde G_2^{\delta}.
\]

This completes the proof.
\end{proof}

\subsection{$\T_N$-Configurations and Condition $O_N$} Let $m,n \ge 1$, let $\Gamma$ be defined as in \eqref{set-G}, and assume that $N \ge 2$.

We begin by recalling the definition of $\T_N$-configurations from \cite{GKY26}.

\begin{definition} \label{T-N} 
Let $\xi_1,\dots,\xi_N\in\R^{m\times n}\times \R^{m\times n}.$ The  $N$-tuple 
$
(\xi_1,\dots ,\xi_N)
$
 is called a  {\em  $\T_N$-configuration}   if  there exist 
 \[
 \rho\in \R^{m\times n}\times \R^{m\times n}, \quad \gamma_1,\dots,\gamma_N\in \Gamma, \quad \kappa_1,\dots,\kappa_N>1
 \]
  such that $
 \gamma_1+\gamma_2+\dots+\gamma_N=0$ and
\begin{equation}\label{t-N}
\xi_1=\rho+\kappa_1\gamma_1,  \quad  \xi_i=\rho+\gamma_1+\dots+\gamma_{i-1}+\kappa_i \gamma_i  \quad \forall\, 2 \le i\le N, 
\end{equation}

In this case, define  
\[
\pi_1=\rho, \quad \pi_i=\rho+\gamma_1+\dots  +\gamma_{i-1}  \quad \forall\,  2\le i\le   N, 
\]
 and set
\[
\T(\xi_1,\dots,\xi_N)=\bigcup_{j=1}^N [\xi_j,\pi_j]. 
\]
\end{definition}

\begin{remark}\label{rk-3} 
\leavevmode

(i) Let $\chi_i=1/\kappa_i\in (0,1)$ for $1\le i\le N,$  and  $\pi_{N+1}=\pi_1=\rho.$ Then
\[
 \pi_{k+1}=  \chi_k \xi_k+(1-\chi_k)\pi_k
 \quad \forall\, 1\le k \le N.
\]
We extend the index $k$ in all quantities periodically modulo $N$.

(ii) In general, let $P_k,X_k\in\R^d$ and  $ t_k\in (0,1)$ be periodic in  $k$ modulo  $N$ and satisfy  
\[
 P_{k+1}=t_kX_k+(1-t_k)P_k\quad  \forall\;k \; \mod N.
\]
 Then, a simple cyclic computation shows that 
\[
 P_i=\sum_{j=1}^N \nu_{i}^{j}X_{j} \quad  \forall\;i \; \mod N,
 \]
 where the coefficients $\nu_i^j$  (indices taken modulo $N$) are given  by
\begin{equation}\label{co-eff-0}
 \begin{split}    \nu_{i}^{i-1} & =\frac{t_{i-1}}{1-(1-t_1)\cdots(1-t_N)},\\
  \nu_{i}^{j} &  =\frac{(1-t_{i-1})(1-t_{i-2})\cdots (1-t_{j+1})t_{j}}{1-(1-t_1)\cdots(1-t_N)}\quad  \forall\, i-N\le j\le i-2. \end{split}
  \end{equation}
  
Note that the coefficients $\nu_i^j$ satisfy
\begin{equation}\label{co-eff-00}
\sum_{j=1}^N \nu_{i}^{j}=1 \quad  \forall\;i \; \mod N.
  \end{equation}
\end{remark}

\medskip

The following key result, similar to \cite[Theorem~3.3]{GKY26},  follows from Lemma~\ref{lem-0}.

\begin{theorem} \label{convex-block-thm} Let   $(\xi_1,\xi_2,\dots  ,\xi_N)$ be a $\T_N$-configuration and 
\[
\eta= \lambda \xi_i +(1-\lambda) \pi_i,
\]
where $1\le i\le N$ and $0\le \lambda\le 1.$   
 Then $  \eta=\sum_{j=1}^N  \nu_j \xi_j,$    where 
   \[
   \nu_i= \lambda +  (1-\lambda) \nu_{i}^i, \quad \nu_j=(1-\lambda)  \nu_{i}^j \quad (j\ne i),
   \] 
with $\{\nu_{i}^j\}_{j=1}^N$  given by   \eqref{co-eff-0} and  $t_k=\chi_k$ for all $1\le k\le N$.

Moreover, for any bounded open set $G\subset \R^{n}$  and   $0<\delta<1,$  there exists  a  function  $(\varphi,  \Psi) \in C^\infty_0(G; \R^m\times  \R^{m\times n})$   satisfying the following properties:
\begin{itemize}
\item[(a)] $\dv\Psi=0$ and $\eta+(D\varphi,  \Psi)\in  [\T(\xi_1,\dots,\xi_N)]_\delta$ in $G;$
\item[(b)]  $\|\varphi\|_{L^\infty(G)} <\delta;$
\item[(c)] there exist pairwise disjoint open  sets $G_1,\dots  ,G_N\subset\subset G$ such that
\[
 \eta+ (D\varphi,  \Psi)=\xi_j \; \text{ in } \, G_j, \quad  |G_j|\ge (1-\delta) \nu_j |G|  \quad \forall\, 1\le j\le N;
\]
in particular, $|\cup_{j=1}^N G_j|=\sum_{j=1}^N  |G_j|\ge(1-\delta)|G|.$
\end{itemize}
 \end{theorem}

\begin{proof}
  If $\lambda=1$, then $ \eta=\xi_i=  \sum_{j=1}^N  \nu_j \xi_j,$ where $\nu_i=1$ and $\nu_j=0$ for all $j\ne i.$ Moreover, (a)-(c) hold with 
  $(\varphi,  \Psi)\equiv (0,0),$ for some $G_i\subset\subset G$  and $G_j=\emptyset$ for all $j\ne i$ in (c).

Now  assume $0\le \lambda<1.$
 Without loss of generality,  we assume $i=1;$ namely, $\eta= \lambda \xi_1 +(1-\lambda) \pi_1.$

 Since $\pi_{k+1}=\chi_k\xi_k+(1-\chi_k)\pi_k$ for all $k \,\mod N$, by Remark \ref{rk-3},  we have  $\pi_1= \sum_{j=1}^N  \nu_{1}^j \xi_j$ and  thus
\[
 \eta= ( \lambda +  (1-\lambda) \nu_{1}^1  )\xi_1+ \sum_{j=2}^N  (1-\lambda)  \nu_{1}^j \xi_j =\sum_{j=1}^N \nu_j \xi_j,
 \]
where $\nu_1^j$ is defined by (\ref{co-eff-0}) with $t_k=\chi_k.$
Note that, with  $\tau=(1-\chi_1)\dots(1-\chi_N),$
\begin{equation}\label{co-eff-1}
\chi_N=\nu_1^{N}(1-\tau), \;\;  \chi_{N-j+1}(1-\chi_{N-j+2})\cdots   (1-\chi_N)  =\nu_{1}^{N-j+1}(1-\tau)\;\; (2\le j\le N).
\end{equation}

Let $G\subset \R^{n}$ be a bounded open set  and   $0<\delta<1$.  Assume that $\epsilon \in (0,\delta)$ is a number to be chosen later.

  Since $\gamma= \xi_1-\pi_1  = \kappa_1\gamma_1\in\Gamma,$ we have that $\eta+(1-\lambda) \gamma=\xi_1$ and $\eta-\lambda \gamma=\pi_1.$
Thus, applying Lemma \ref{lem-0} with  $0\le \lambda<1$,  $\gamma \in\Gamma$,  $G\subset \R^{n}$ and $\epsilon \in (0,\delta)$,  we obtain  a smooth function  $(\varphi_0,\Psi_0)$ compactly supported in  $G_0=G$, satisfying (a)--(c) of the lemma.  Hence
\begin{equation}\label{start-0}
\begin{cases}
\mbox{$\dv\Psi_0=0$ and $\eta+(D\varphi_0,\Psi_0) \in [\xi_1,\pi_1]_{\epsilon}$ in $G_0;$}\\
\|\varphi_0\|_{L^\infty(G_0)} <\epsilon;\\
 \eta+(D\varphi_0,\Psi_0) =\eta+(1-\lambda) \gamma=\xi_1\; \mbox{in $G'_0;$}\\
 \eta+(D\varphi_0,\Psi_0) =\eta-\lambda \gamma=\pi_1\; \mbox{in $G_0'',$}
\end{cases}
\end{equation}
 where $G_0',G_0''\subset\subset G_0$  are pairwise disjoint open sets, with
\begin{equation}\label{start-01}
  |G_0'|\ge (1-\epsilon) \lambda |G_0|, \quad   |G_0''|\ge (1-\epsilon)(1-\lambda) |G_0|.
\end{equation}

 Let  $\tilde G_1=G_{11}=G_{0}''$. Note that 
 \[
 \eta+(D\varphi_0,\Psi_0) =\pi_1=\pi_{N+1}={\chi_N}\xi_N+(1-{\chi_N})\pi_N.
 \]
  Applying Lemma \ref{lem-0}    with 
  \[
  \lambda= \chi_N\in (0,1), \quad \gamma=\xi_N-\pi_N=\kappa_N\gamma_N\in\Gamma, \quad G=G_{11}, \quad 
  \epsilon\in(0,\delta),
  \]
    we obtain a smooth  function $(\varphi_{11},\Psi_{11})$ compactly supported in $ G_{11}$ and satisfying (a)--(c) of the lemma. Hence,
  $\pi_1+(D\varphi_{11},\Psi_{11})  \in [\xi_N,\pi_N]_{\epsilon}$ in $ G_{11}$, and
  \[
 \mbox{$\pi_1+(D\varphi_{11},\Psi_{11})  =\xi_N$ \, in $ G_{11}'$, \; \;  $\pi_1+(D\varphi_{11},\Psi_{11})=\pi_N$ \, in $ G_{11}'',$}
  \]
   where  $G_{11}',G_{11}''\subset\subset G_{11}$ are disjoint  open sets  with
\[
| G_{11}'|\ge (1-\epsilon) {\chi_N} | G_{11}|,  \quad   | G_{11}''|\ge (1-\epsilon){(1-\chi_N)} | G_{11}|.
\]

Let $ G_{12}= G_{11}'',$ in which $\pi_1+(D\varphi_{11},\Psi_{11})=\pi_N={\chi_{N-1}}\xi_{N-1}+(1-{\chi_{N-1}})\pi_{N-1}.$ Thus,  repeating the above procedure, we obtain  a smooth  function  $(\varphi_{12},\Psi_{12})$ compactly supported in $G_{12}$ such that
\[
\begin{cases} \pi_N+(D\varphi_{12},\Psi_{12}) \in [\xi_{N-1},\pi_{N-1}]_{\epsilon} \;\; \text{ in $G_{12}$,}\\
\pi_N+(D\varphi_{12},\Psi_{12}) =\xi_{N-1} \;\; \text{in $G_{12}'$,} \quad 
\pi_N+(D\varphi_{12},\Psi_{12}) =\pi_{N-1} \;\; \text{in $G_{12}'',$}\end{cases}
 \]
  where $G_{12}',G_{12}''\subset\subset G_{12}$ are disjoint open sets  with
\[
\begin{split} & |G_{12}'|  \ge (1-\epsilon) {\chi_{N-1}} | G_{12}| \ge (1-\epsilon)^2  {\chi_{N-1}} (1-\chi_N) | G_{11}|,
\\
&| G_{12}''|\ge (1-\epsilon)(1- \chi_{N-1}) | G_{12}|\ge (1-\epsilon)^2  (1-\chi_{N-1})(1-  \chi_N ) | G_{11}|.\end{split}
\]

Continuing  this procedure $N$ times, we  obtain smooth functions  $(\varphi_{1j},\Psi_{1j})$  compactly supported in $ G_{1j}= G_{1(j-1)}''$ for $1\le j\le N$, where $G_{10}''=G_0''$,  such that
\begin{equation}\label{scheme-1j}
\begin{cases}
\mbox{$\dv\Psi_{1j}=0$ and $\pi_{N-j+2}+(D\varphi_{1j},\Psi_{1j}) \in [\xi_{N-j+1},\pi_{N-j+1}]_{\epsilon}$ in $G_{1j};$}\\
\| \varphi_{1j}\|_{L^\infty( G_{1j})} <\epsilon;\\
\mbox{$\pi_{N-j+2}+(D\varphi_{1j},\Psi_{1j}) =\xi_{N-j+1}$ in $G_{1j}'$;}\\
\mbox{$\pi_{N-j+2}+(D\varphi_{1j},\Psi_{1j})=\pi_{N-j+1}$ in $G_{1j}'' $,}
 \end{cases}
\end{equation}
where  $G_{1j}',G_{1j}''\subset\subset G_{1j}$ are disjoint  open sets  with
\begin{equation}\label{sm-0}
\begin{split} & |G_{1j}'|  \ge  (1-\epsilon)^{j}  \chi_{N-j+1}(1-\chi_{N-j+2})\cdots   (1-\chi_N)  | G_{11}|,
\\
&| G_{1j}''|\ge  (1-\epsilon)^{j} (1-\chi_{N-j+1})\cdots   (1-\chi_N)  | G_{11}|.
\end{split}
\end{equation}

Define   $(\varphi_1,\Psi_1)\in C^\infty_0(\tilde G_1; \R^m\times  \R^{m\times n})$  by setting
\[
 \varphi_1=\varphi_{11}+\dots  +\varphi_{1N}, \quad \Psi_1=\Psi_{11}+\dots  +\Psi_{1N}.
\]
Then  
\begin{equation}\label{scheme-0}
\begin{cases}
 \mbox{$\dv \Psi_1=0$ and  $\pi_1+(D \varphi_1,\Psi_1)\in [\T(\xi_1,\dots  ,\xi_N)]_{\epsilon}$ in $\tilde G_1;$} \\
 \mbox{$\| \varphi_1\|_{L^\infty(\tilde G_1)} <N\epsilon;$}\\
\mbox{$\pi_1+(D \varphi_1,\Psi_1) = \xi_{N-j+1}$ in $G_{1j}' \quad\forall\, 1\le j\le N;$}\\
 \mbox{$(D \varphi_1,\Psi_1) =0$ in $G_{1N}'',$}
 \end{cases}
\end{equation}
where, by (\ref{co-eff-1}) and (\ref{sm-0}),  $G_{11}',\dots  , G_{1N}',G_{1N}''\subset\subset\tilde G_1$ are  pairwise disjoint open and \begin{equation}\label{meas-G-j}
\begin{cases} | G'_{1j}|\ge (1-\epsilon)^j  (1-\tau)\nu_{1}^{N-j+1}  |\tilde G_1| \quad  (1\le j\le N), \\
 | G_{1N}''|\ge  (1-\epsilon)^{N} \tau |\tilde G_1|.
 \end{cases}
\end{equation}

We now iterate  the  scheme (\ref{scheme-0})   inductively to obtain  open sets
\[
 \tilde G_{k}=  G_{ (k-1)N}'', \quad G'_{k1},\dots, G'_{ kN}, \; G_{kN}''\subset\subset\tilde G_{k} \;\; \text{all disjoint,}
\]
 and smooth  functions  $(\varphi_{k}, \Psi_{k})$ compactly supported in  $\tilde G_{k},$ for all $k=1,2,\dots,$ such that
 \begin{equation}\label{scheme-1}
\begin{cases}
\mbox{$\dv \Psi_{k} =0$ and $\pi_1+(D \varphi_{k},\Psi_{k})\in [\T(\xi_1,\dots  ,\xi_N)]_{\epsilon}$ in $\tilde G_{k};$}\\
\| \varphi_k\|_{L^\infty(\tilde G_{k})} <N\epsilon;\\
\mbox{$\pi_1+(D \varphi_{k},\Psi_{k})= \xi_{N-j+1}$ in $G_{kj}'$, and} \\
 |G'_{kj}|\ge (1-\epsilon)^j  (1-\tau)\nu_{1}^{N-j+1}  | \tilde G_{k}|
\quad \forall j=1,\dots  ,N;\\
 \mbox{$(D \varphi_{k},\Psi_{k})=0$ in $G_{kN}''$, \; $|G_{kN}''|\ge (1-\epsilon)^{N} \tau | \tilde G_k|.$}
\end{cases}
\end{equation}

Note that, for $k=1,2,\dots ,$
\[
 |\tilde G_k|  =|G_{(k-1)N}''|\ge (1-\epsilon )^{N} \tau |\tilde G_{k-1}|\ge\dots  \ge  (1-\epsilon)^{(k-1)N} \tau^{k-1} |\tilde G_1|.
\]
Hence,    for all $k=1,2,\dots  $ and $1\le j\le N,$
\begin{equation}\label{meas-G}
 |G'_{kj}|   \ge (1-\epsilon)^j  (1-\tau)\nu_{1}^{N-j+1}  | \tilde G_{k}|
   \ge (1-\epsilon)^{j+(k-1)N} (1-\tau)\nu_{1}^{N-j+1}   \tau^{k-1}  |\tilde G_1|.
      \end{equation}

 Let  $\ell\ge 2$  be an integer to be  chosen later.  Define
\[
(\tilde\varphi,\tilde\Psi) =\sum_{k=1}^\ell ( \varphi_{k}, \Psi_{k}), \quad V_j = \cup_{k=1}^\ell G'_{ kj}  \quad \forall\,1\le j\le N.
\]
Then $(\tilde\varphi,\tilde\Psi)$ is smooth and compactly supported in $\tilde G_1$ and satisfies that
\[
\begin{cases}
\mbox{$\dv\tilde\Psi=0$ and  $\pi_1+(D\tilde \varphi, \tilde \Psi)\in [\T(\xi_1,\dots  ,\xi_N)]_{\epsilon}$ in $\tilde G_{1}$;}
\\
\|\tilde\varphi \|_{L^\infty(\tilde G_{1})} <N \ell \epsilon;  \\
  \mbox{$\pi_1 +  (D \tilde  \varphi, \tilde \Psi )= \xi_{N-j+1}$ in $V_j$} \quad \forall\, 1\le j\le N, 
    \end{cases}
  \]
where, by (\ref{start-01}) and (\ref{meas-G}), for each $ 1\le j\le N,$
\begin{equation}\label{scheme-3}
\begin{split}  |V_j|    &= \sum_{k=1}^\ell  | G'_{ kj} |   \ge \sum_{k=1}^\ell (1-\epsilon)^{j+(k-1)N} (1-\tau)\nu_{1}^{N-j+1}   \tau^{k-1}  |\tilde G_1|
 \\
 & \ge (1-\epsilon)^{ \ell N}  (1-\tau)\nu_{1}^{N-j+1}  \Big (\sum_{k=1}^\ell   \tau^{k-1} \Big)  |\tilde G_1|   \\
 &= (1-\epsilon)^{ \ell N} (1-\tau^{\ell})\nu_{1}^{N-j+1} |\tilde G_1| \\ &\ge  (1-\epsilon)^{1+ \ell N} (1-\tau^{\ell})\nu_{1}^{N-j+1} (1-\lambda) | G|.
  \end{split}
\end{equation}

We  select  a  sufficiently large $\ell\ge 2$ such that
\begin{equation}\label{lg-1}
1-\tau^{\ell }\ge (1-\delta)^{1/2},
\end{equation}
and  a sufficiently small $ \epsilon \in (0,\delta)$ such   that
\begin{equation}\label{sm-1}
(1+N\ell)\epsilon <\delta, \quad  (1-\epsilon)^{1+ \ell N} \ge (1-\delta)^{1/2}.
\end{equation}

Define
\[
\begin{cases} (\varphi,\Psi)=(\tilde \varphi,\tilde \Psi)+(\varphi_0,\Psi_0),\\
G_1=G'_0 \cup V_N, \;\;\; G_j=V_{N-j+1} \quad  \forall\, 2\le j\le N.
\end{cases}
\]
Then  $(\varphi,\Psi)$ is smooth and compactly supported in $G$, and satisfies
\[
\begin{cases}
 \text{$\dv\Psi=0$ and $\eta+(D \varphi,\Psi)\in [\T(\xi_1,\dots  ,\xi_N)]_{\epsilon} \subset [\T(\xi_1,\dots  ,\xi_N)]_{\delta}$ in $G;$}\\
\| \varphi \|_{L^\infty(G)} <(1+N\ell)\epsilon<\delta;\\
 \mbox{$\eta+ (D \varphi ,\Psi )  =\xi_j$ in $G_j$} \quad \forall\,   1\le j\le N. 
  \end{cases}
\]

Moreover, from   (\ref{start-01}), (\ref{scheme-3}), (\ref{lg-1}) and (\ref{sm-1}), it follows that
\[
\begin{split} |G_1| =|G_0'|+|V_N| & \ge (1-\epsilon) \lambda |G|+(1-\epsilon)^{1+\ell N} (1-\tau^{\ell})\nu_{1}^{1} (1-\lambda) |G|\\
&\ge (1-\delta) ( \lambda +(1-\lambda)\nu_1^1)|G| =(1-\delta) \nu_1|G|;
\\
 |G_j| = |V_{N-j+1}| &\ge  (1-\epsilon)^{1+\ell N} (1-\tau^{\ell})\nu_{1}^{j} (1-\lambda) |G|\\
&\ge (1-\delta)  (1-\lambda)\nu_1^j |G| =(1-\delta) \nu_j |G|  \quad \forall\,2\le j\le N.
\end{split}
\]
Therefore, assertions (a)--(c) of the theorem are satisfied.   

This completes the proof.
\end{proof}

%%%%%%%%%%%%%%%%%%%%%%%

In what follows, for $\rho \in \mathbb{R}^{m \times n} \times \mathbb{R}^{m \times n}$, we write
\[
\rho = (\rho^1, \rho^2), \; \text{ with } \, \rho^1, \rho^2 \in \mathbb{R}^{m \times n}.
\]
For $r > 0$, let $\mathbb{B}_r = \mathbb{B}_r(0)$ and $\bar{\mathbb{B}}_r = \bar{\mathbb{B}}_r(0)$ denote the open and closed balls of radius $r$ centered at $0$ in $\mathbb{R}^{m \times n} \times \mathbb{R}^{m \times n}$.

We next recall the definition of Condition~$O_N$ from \cite{GKY26}.

\begin{definition}\label{O-N}
A function $\sigma \colon \mathbb{R}^{m \times n} \to \mathbb{R}^{m \times n}$
is said to satisfy \emph{Condition $O_N$} if there exist constants
\[
r_0 > 0, \quad 0 < \delta_1 < \delta_2 < 1,
\]
and continuous functions
\[
(\kappa_i,\gamma_i) \colon \overline{\mathbb{B}}_{r_0}
\to [1/\delta_2,\,1/\delta_1] \times \Gamma,
\qquad 1 \le i \le N,
\]
satisfying $\gamma_1(\rho) + \cdots + \gamma_N(\rho) = 0$
 for all $\rho \in \overline{\mathbb{B}}_{r_0},$  such that, if
\[
\begin{cases}
\pi_1(\rho) = \rho, \\[0.2em]
\pi_i(\rho) = \rho + \gamma_1(\rho) + \cdots + \gamma_{i-1}(\rho),
\qquad 2 \le i \le N, \\[0.3em]
\xi_i(\rho) = \pi_i(\rho) + \kappa_i(\rho)\gamma_i(\rho),
\qquad 1 \le i \le N,
\end{cases}
\quad \rho \in \overline{\mathbb{B}}_{r_0},
\]
then the following properties hold:
\begin{itemize}
\item[(P1)]
\begin{enumerate}
\item[(i)]
$\xi_i(\rho) \in \mathcal{K}$ for all $\rho \in \overline{\mathbb{B}}_{r_0}$
and $1 \le i \le N$, where $\mathcal{K}$ denotes the graph of $\sigma$.
\smallskip
\item[(ii)] $\xi_i^1(\overline{\mathbb{B}}_{r_0}) \cap \xi_j^1(\overline{\mathbb{B}}_{r_0})
= \pi_i^1(\overline{\mathbb{B}}_{r_0}) \cap \pi_j^1(\overline{\mathbb{B}}_{r_0})
= \emptyset$ for all $ 1 \le i \ne j \le N.$
\end{enumerate}

\medskip

\item[(P2)]
For $1 \le i \le N$, $0 < r \le r_0$, and $0 \le \lambda \le 1$, define
\[
S_i^r(\lambda)
:= \bigl\{ \lambda \xi_i(\rho) + (1-\lambda)\pi_i(\rho)
: \rho \in \mathbb{B}_r \bigr\}.
\]
Then, for each $0 < r \le r_0$ and
$\lambda \in \{0\} \cup [\delta_2,1)$, the family
\[
\{ S_i^r(\lambda) \}_{i=1}^N
\]
consists of open and pairwise disjoint sets.
\end{itemize}

In this case, for all $1 \le i \le N$, $0 < r \le r_0$, and $0 \le \lambda \le 1$, define
\[
L_i^r(\lambda)
:= \bigcup_{\substack{\alpha \in S_i^r(\lambda),\; \beta \in S_i^r(0)\\
\alpha - \beta \in \Gamma}}
[\alpha,\beta],
\quad
S^r(\lambda) := \bigcup_{i=1}^N S_i^r(\lambda),
\quad
L^r(\lambda) := \bigcup_{i=1}^N L_i^r(\lambda),
\]
and, for $\delta_2 \le \mu < 1$, define
\[
\Sigma^r(\mu) := \bigcup_{\delta_2 \le \lambda \le \mu} L^r(\lambda).
\]
Finally, set
\[
\Sigma(1) := \bigcup_{\delta_2 \le \mu < 1} \Sigma^{r_0}(\mu).
\]
\end{definition}

 \begin{remark}\label{remk33}    Let $\sigma$ satisfy Condition $O_N$ as above. We summarize some key consequences  established in \cite{GKY26}.

 \medskip
 
(i)  For each $1\le i\le N,$ $0\le\lambda\le 1$, and $\rho\in\bar\B_{r_0},$ define
\begin{equation}\label{chi-i}
 \chi_i(\rho)= {1}/{\kappa_i(\rho)},\quad \zeta_i(\lambda,\rho) =\lambda \xi_i(\rho) +(1-\lambda)\pi_i(\rho).
\end{equation}
Then  $\delta_1\le \chi_i(\rho) \le \delta_2$,  and
\[
\pi_{i+1}(\rho)=\chi_i(\rho)\xi_i(\rho)+(1-\chi_i(\rho))\pi_i(\rho).
\]

Moreover, for all  $0<r\le r_0$ and $\rho\in \B_r,$  
\[
\zeta_i(\lambda,\rho)\in S^r_i(\lambda),\quad \pi_i(\rho)\in S^r_i(0),\quad \zeta_i(\lambda,\rho)-\pi_i(\rho)=\lambda\kappa_i(\rho)\gamma_i(\rho) \in \Gamma,
\]
and hence,
\[
[\zeta_i(\lambda,\rho),\, \pi_i(\rho)]  \subset L^r_i(\lambda).
\]

\medskip

(ii)  Let  $\delta_2<\lambda\le 1,$ $0<r\le r_0$, and $\rho\in \B_{r}.$ Set 
$X_i=\zeta_i(\lambda,\rho)$, defined in \eqref{chi-i},  for $1\le i\le N.$   Then  
\begin{equation}\label{eq-0} 
\pi_i(\rho) =\sum_{j=1}^N  \nu_i^j(\lambda,\rho) X_j \qquad \forall \, 1\le i\le N,
\end{equation}
where the coefficients $\nu_i^j(\lambda,\rho) = \nu_i^j$, for $1 \le i,j \le N$, are determined by \eqref{co-eff-0} with constants
 \[
 t_k= \frac{\chi_k(\rho)}{\lambda} \qquad \forall \, 1\le k\le N.
 \]
 
 Since ${\delta_1}/{\lambda}\le t_k\le   {\delta_2}/{\lambda}<1$ for all  $1\le k\le N,$   it follows that,  for each $1\le i\le N,$
\begin{equation}\label{eq-1}
\sum_{j=1}^N  \nu_i^j(\lambda,\rho) =1,  \quad  \nu_i^j(\lambda,\rho)  \ge  (\lambda-\delta_2)^{N-1}\delta_1 \qquad \forall\, 1\le  j \le N.
\end{equation}

Moreover,   the $N$-tuple $(X_1,\dots,X_N)$ is a $\T_N$-configuration, and
\[
\T(X_1,\dots,X_N)\subset L^r(\lambda).
\]

\medskip

(iii)   For each $1\le i\le N$, 
 $\delta_2\le \lambda<1,$ and $0< r\le r_0,$ the sets 
$
 L^r_i(\lambda)$, $L^r(\lambda)$, $\Sigma^r(\lambda)$, and $ \Sigma(1)$  are all open. 
 
 Moreover, if  $1\le k\le N$ and $0< r<s \le r_0,$ then
\[
\overline{S_k^r(\lambda)}\subset S_k^s(\lambda) \quad\forall\, 0 \le\lambda\le 1, \quad \overline{\Sigma^r(\lambda)}\subset \Sigma^s(\lambda) \quad\forall\, \delta_2\le \lambda<1.
\]

Furthermore,  if $\Delta\subset \Sigma(1)$ is  compact, then there exist $r\in (0, r_0)$ and $\lambda\in [\delta_2,1)$ such that 
\[
\Delta\subset \Sigma^{r}(\lambda).
\]

 \end{remark}

\medskip

The  key steps of  our convex integration scheme rely on the following result, similar to \cite[Theorem 3.6]{GKY26}.

\begin{theorem}\label{lem1} Let $\sigma$ satisfy Condition $O_N.$ Let $\delta_2\le \lambda \le \mu<1,$ $0<r\le r_0,$ and   
\[
Y =q\, \zeta_i(\lambda', \rho)+(1-q)\, \pi_i(\rho')\in \Sigma^r(\lambda),
\]
where $q\in [0,1],$ $1\le i\le N,$ $\delta_2 \le \lambda' \le \lambda,$ and $\rho, \rho' \in \B_r$ with
$
\zeta_i(\lambda', \rho)- \pi_i(\rho')\in\Gamma.
$
Let
\[
 X_j=\zeta_j(\mu, \rho), \quad X'_j=\zeta_j(\mu, \rho') \quad\forall\, 1\le j\le N.
\]

Then,  for any bounded open set $G\subset \R^{n}$  and  $0<\tau<1,$ there exists  a function $(\varphi, \Psi)\in C_0^\infty (G;\R^m \times \R^{m\times n})$   with  the following properties:
\begin{itemize}
\item[(a)] $\dv\Psi=0$ and $Y+(D\varphi,  \Psi)\in \Sigma^r(\mu)$ in $G;$
\item[(b)]  $\|\varphi\|_{L^\infty(G)} <\tau;$
\item[(c)]   there exist pairwise disjoint open  sets $G'_1,\dots  ,G'_N,  G''_1,\dots, G''_N\subset\subset G$ such that
\begin{equation}\label{eq-d}
\begin{cases}
 Y+(D\varphi,  \Psi)=\chi_{G'_j} X_j+\chi_{G''_j} X_j'\in S_j^r(\mu) \;\; \text{in $G_j=G'_j\cup G''_j$} & \forall\, 1\le j\le N,\\[2mm]
|G_i|\ge  (1-\tau) \big [(1-q) \nu_i^i(\mu,\rho') +q  \big (\frac{\lambda'}{\mu} +(1-\frac{\lambda'}{\mu})  \nu_i^i(\mu,\rho)\big )\big  ] |G|,\\[1mm]
|G_j| \ge (1-\tau)  \big [(1-q) \nu_i^j(\mu,\rho') +q  (1-\frac{\lambda'}{\mu})  \nu_i^j(\mu,\rho) \big ] |G| &\forall\, j\ne i,
\end{cases}
\end{equation}
where $\nu_i^j(\mu,\rho)\; (1\le i,j\le N)$  are defined in \eqref{eq-0}.
Moreover, \[ \sum_{j=1}^N  |G_j|\ge (1-\tau)|G|.\]
\end{itemize}

 {\rm (See Figure \ref{fig1} for an illustration in the case  $N=5$.)}
  \end{theorem}

   \begin{figure}[ht]
\begin{center}
\begin{tikzpicture}[scale =1]
\draw[thick] (-5,-2.2)--(1.1,-0.9);
   \draw[](-5.2,-2) node[below]{$X'_1$};
      \draw[](-4.6,-2.1) node[below]{$X_1$};
    \draw[](0.2,5.9) node[above]{$X'_3$};
        \draw[](0.5,6.6) node[left]{$X_3$};
              \draw[](1,-3.2) node[right]{$X'_5$};
       \draw[](0.4,-3.4) node[right]{$X_5$};
   \draw(3.9,1.6) node[right]{$X_4$};
     \draw[](3.8,2.1) node[right]{$X'_4$};
    \draw[](-6.3,2.25) node[above]{$X'_2$};
        \draw[](-6.8,2.3) node[below]{$X_2$};
      \draw[thick] (1.1,-0.9)--(0.9,3.1);
  \draw[thick] (1.1,-0.9)--(1.205,-3);
    \draw[ thick] (-5,1.2)--(-2.1,4.1);
     \draw[thick] (-1.95,-1.55)--(-6.3,1.2+2.75*1.3/3.05);
     \draw[ thick] (-0.1,6.1)--(-2.1,4.1);
      \draw[ thick] (-2.1,4.1)--(0.9,3.1);
          \draw[thick] (0.9,3.1)--(3.9,2.1);
       \draw(-1.1,5.4) node[left]{$\zeta_3(\lambda',\rho)$};
     \draw[thick] (-2.2,4.3)--(4,1.7);
       \draw[thick] (-2.2,4.3)--(-5.5,1.1);
         \draw[thick] (-2.2,4.3)--(0,4.3+6.4/3);
        \draw[thick ] (-2,-1.8)--(-5.5,1.1);
          \draw[thick ] (-6.5,-1.8+26.1/7)--(-5.5,1.1);
        \draw[thick ] (-2,-1.8)--(0.8,-1.4);
               \draw[thick ] (-2,-1.8)--(-4.8,-2.2);
                \draw[thick ] (0.9,3)--(0.8,-1.4);
                 \draw[thick ] (0.76,-3.16)--(0.8,-1.4);
   \draw[fill] (-5.5,1.1)  circle (0.05);
            \draw(-5.7,1.1) node[below]{$\pi_3(\rho)$};
                    \draw(-4.9,1.2) node[right]{$\pi_3(\rho')$};
            \draw[fill]  (-5,1.2) circle (0.05);
              \draw[ultra thick] (-1.1,5.35)--(-5,1.2);
                \draw[fill]  (-1.1, 5.35) circle (0.05);
         \draw(-3.1,3.1) node[right]{$Y$};
           \draw[fill]  (-3.05,3.28) circle (0.05);
\end{tikzpicture}
\end{center}
\caption{Illustration of Theorem~\ref{lem1} for $N=5$.  
Here $i=3$, and $Y$ lies on the (ultra-thick) line segment $[\zeta_3(\lambda',\rho), \pi_3(\rho')]$.
The sets $\T(X_1,\dots,X_N)$ and $\T(X'_1,\dots,X'_N)$ are represented by the corresponding  line segments.
The points $X_j, X'_j$ lie in $S^r_j(\mu)$, and all segments shown  lie in $\Sigma^r(\mu)$.
Only these points and segments are used in the proof.}
\label{fig1}
\end{figure}
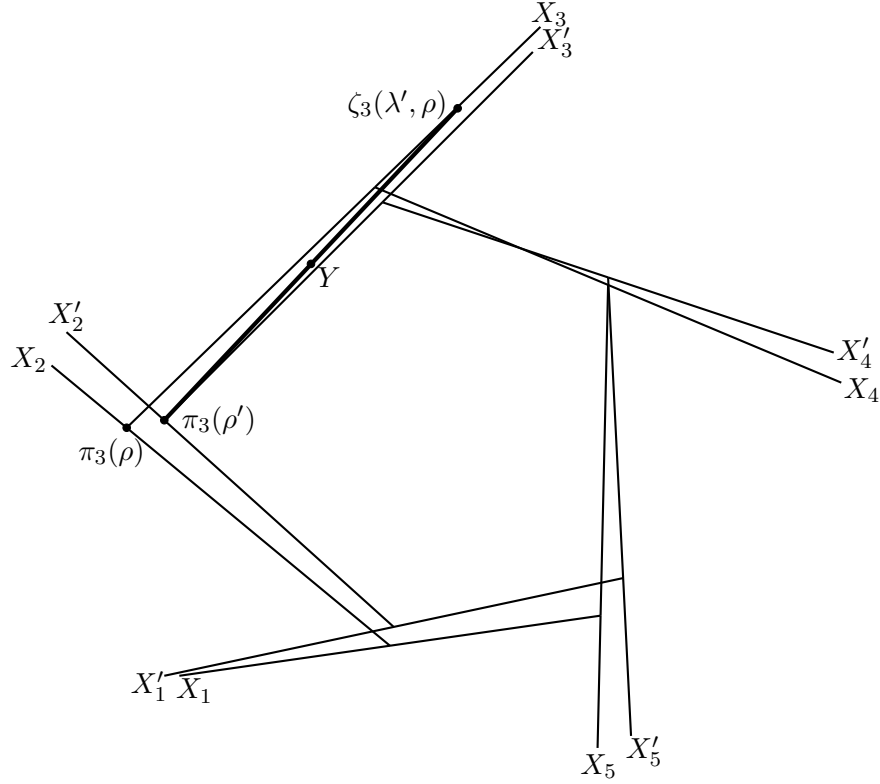

\begin{proof}  

1. Since $[\zeta_i(\lambda', \rho),\, \pi_i(\rho')]\subset L^r_i(\lambda')\subset \Sigma^r(\mu)$ and $\Sigma^r(\mu)$ is open, there exists   $\epsilon_1>0$ such that
\[
[\zeta_i(\lambda', \rho),\, \pi_i(\rho')]_{\epsilon_1}\subset \Sigma^r(\mu).
\]
In what follows, let $ 0< \delta<1$ be a number to be determined later.

Apply Lemma \ref{lem-0} with  $\gamma=\zeta_i(\lambda', \rho)- \pi_i(\rho'),$ $\lambda=q$  and $\epsilon=\delta$ to obtain   $(\varphi_1, \Psi_1 )\in C_0^\infty (G;\R^m \times \R^{m\times n})$   with  the following properties:
\begin{equation}\label{eq-d-0}
\begin{cases}
\text{$\dv\Psi_1=0$  and $Y+(D\varphi_1,  \Psi_1)\in [\zeta_i(\lambda', \rho),\, \pi_i(\rho')]_\delta$ in $G;$}\\
\text{$\|\varphi_1\|_{L^\infty(G)} <\delta;$}\\
\text{there exist disjoint open  sets $G', G''\subset\subset G$ such that}\\
$$
\begin{cases}
\mbox{$Y+(D\varphi_1,  \Psi_1)=\zeta_i(\lambda', \rho)$\, in $G'$,} &|G'|\ge  (1-\delta) q\, |G|,\\
\mbox{$Y+(D\varphi_1,  \Psi_1)=\pi_i(\rho')$\, in $G''$,} & |G''|\ge  (1-\delta)  (1-q) \,|G|.
\end{cases}
$$
\end{cases}
\end{equation}

2. If $G''=\emptyset$ (necessarily, $q=1$),  then we omit this step. Assume $G''\ne \emptyset.$ Note that by \eqref{eq-0},
\[
\pi_i(\rho') = \sum_{j=1}^N  \nu_i^j(\mu,\rho') X'_j.
\]
Since  $(X'_1,\dots ,X'_N)$ is  a $\T_N$-configuration with $\T(X'_1,\dots ,X'_N) \subset  \Sigma^r(\mu),$  by the openness of  $\Sigma^r(\mu),$  there exists $\epsilon_2>0$ such that
\[
  [\T(X'_1,\dots ,X'_N)]_{\epsilon_2} \subset \Sigma^r(\mu).
\]
Apply Theorem \ref{convex-block-thm}  with $(\xi_1,\ldots,\xi_N)=(X'_1,\dots ,X'_N)$ and  $\eta=\pi_i(\rho')$ to obtain  $(\varphi_2, \Psi_2)\in C_0^\infty (G'';\R^m \times \R^{m\times n})$   with  the following properties:
\begin{equation}\label{eq-d-1}
\begin{cases}
\text{$\dv\Psi_2=0$ and $\pi_i(\rho')+(D\varphi_2,  \Psi_2)\in [\T(X'_1,\dots ,X'_N)]_\delta$ in $G'';$}\\
\text{$\|\varphi_2\|_{L^\infty(G'')} <\delta;$}\\
\text{there exist pairwise disjoint open  sets $G''_1,\dots  ,G''_N\subset\subset G''$ such that}\\
$$
\begin{cases}
\mbox{$\pi_i(\rho')+(D\varphi_2,  \Psi_2)=X'_j$ in $G''_j$}\quad \forall\, 1\le j\le N,\\
|G''_j|\ge  (1-\delta)   \nu^j_i(\mu,\rho') \,|G''|   \quad\forall\, 1\le j\le N.
\end{cases}
$$
\end{cases}
\end{equation}

3. If $G'=\emptyset$  (necessarily, $q=0$),  then we skip this step.  Assume $G'\ne \emptyset$.  Note that by \eqref{eq-0},
\[
\zeta_i(\lambda',\rho) =  \left (\frac{\lambda'}{\mu} +\Big (1-\frac{\lambda'}{\mu}\Big )  \nu^i_i(\mu,\rho)\right )X_i+\sum_{1\le j\le N,\, j\ne i} \Big ( 1-\frac{\lambda'}{\mu}\Big ) \nu_i^j(\mu,\rho) X_j.
\]
   Since  $(X_1,\dots ,X_N)$ is  a $\T_N$-configuration with $\T(X_1,\dots ,X_N) \subset  \Sigma^r(\mu),$  by the openness of  $\Sigma^r(\mu),$  there exists $\epsilon_3>0$ such that
   \[
    [\T(X_1,\dots ,X_N)]_{\epsilon_3} \subset \Sigma^r(\mu).
    \]
Now apply Theorem \ref{convex-block-thm} with $(\xi_1,\ldots,\xi_N)=(X_1,\dots ,X_N)$  and $\eta=\zeta_i(\lambda',\rho)$ to  obtain   $(\varphi_3, \Psi_3)\in C_0^\infty (G';\R^m \times \R^{m\times n})$   with  the following properties:
\begin{equation}\label{eq-d-2}
\begin{cases}
\text{$\dv\Psi_3=0$ and $\zeta_i(\lambda',\rho)+(D\varphi_3,  \Psi_3)\in [\T(X_1,\dots ,X_N)]_\delta$ in $G';$}\\
\text{$\|\varphi_3\|_{L^\infty(G')} <\delta;$}\\
\text{there exist pairwise disjoint open  sets $G'_1,\dots  ,G'_N\subset\subset G'$ such that}\\
$$
\begin{cases}
\mbox{$\zeta_i(\lambda',\rho)+(D\varphi_3,  \Psi_3)=X_j$ in $G'_j$}\quad \forall\, 1\le j\le N,\\
|G'_i|\ge  (1-\delta) \big (\frac{\lambda'}{\mu} +(1-\frac{\lambda'}{\mu})  \nu_i^i(\mu,\rho)\big ) |G'|,\\
|G'_j| \ge  (1-\delta)   ( 1-\frac{\lambda'}{\mu})  \nu_i^j(\mu,\rho) \,|G'| \quad\forall\, j\ne i.
\end{cases}
$$
\end{cases}
\end{equation}

4. Finally, let
\[
 (\varphi,\Psi)= (\varphi_1,\Psi_1 )\chi_{G} + (\varphi_2,\Psi_2 )\chi_{G''} +(\varphi_3,\Psi_3)\chi_{G'} .
\]
Then  $(\varphi, \Psi)\in C_0^\infty (G;\R^m \times \R^{m\times n})$ and, by \eqref{eq-d-0}-\eqref{eq-d-2}, we have 
\[
  \begin{cases}  
\text{$ \dv\Psi=0$ in $G$;}\\
 Y+(D\varphi,  \Psi)  \in [\zeta_i(\lambda', \rho),\, \pi_i(\rho')]_\delta\cup [\T(X'_1,\dots ,X'_N)]_\delta \cup [\T(X_1,\dots ,X_N)]_\delta \text{ in $G$;}\\ 
\|\varphi\|_{L^\infty(G)} <2\delta;\\
$$
\begin{cases}
 Y+(D\varphi,  \Psi)=\chi_{G'_j} X_j+\chi_{G''_j} X_j' \in S_j^r(\mu) \;\; \text{in $G_j=G'_j\cup G''_j$} & \forall\, 1\le j\le N,\\
|G_i|\ge  (1-\delta)^2 \big [(1-q) \nu_i^i(\mu,\rho') +q  \big (\frac{\lambda'}{\mu} +(1-\frac{\lambda'}{\mu})  \nu_i^i(\mu,\rho)\big )\big ] |G|,\\
|G_j| \ge (1-\delta)^2 \big [(1-q) \nu_i^j(\mu,\rho') +q  (1-\frac{\lambda'}{\mu})  \nu_i^j(\mu,\rho) \big ] |G| & \forall\, j\ne i.
\end{cases}
$$
\end{cases}
\]
Consequently,  all requirements (a)--(c) of the theorem will follow  if $0<\delta<1$ is chosen to satisfy
\[
0<\delta<\min\left \{\epsilon_1, \, \epsilon_2,\, \epsilon_3,\, \frac{\tau}{2},\, 1-\sqrt{1-\tau}\right \}.
\]
\end{proof}

%%%%%%%%%%%%%%%%%%%%%%%%%%%%%%

\section{Proof of Theorem \ref{mainthm}}\label{s-4}

In this section, we prove the main Theorem~\ref{mainthm}, following the construction scheme outlined in the Introduction.

\subsection{Main stage theorem}   We first establish a key stage theorem, analogous to \cite[Theorem 4.1]{GKY26}, which plays a central role in the constructions of \eqref{pdr2}.

 \begin{theorem}\label{thm1}  Assume that  $\sigma\colon\R^{m\times n}\to \R^{m\times n} $ satisfies  Condition $O_N,$ that
 $G\subset \R^{n}$ is a bounded open set, and that $u\in C^1(\bar G;\R^m)$ and  $V \in C(\bar G; \R^{m\times n})$ satisfy 
\begin{equation}\label{strict-sub}
 \dv V=0 \; \text{ in $G$,} \quad (Du, V)\in \Sigma^r(\lambda) \; \text{ on $\bar G,$}
\end{equation}
where  $0<r<r_0$ and $\delta_2\le \lambda<1$ are constants.

Then, for any  $\lambda<\mu<1,$  $r<s<r_0$,  and  $0<\epsilon<1,$  there exist a function 
\[
(\tilde u,\tilde V)\in (u,V)+C^{\infty}_{0}(G;\R^m\times\R^{m\times n})
\]
 and finitely many disjoint cubes $Q_1,\ldots,Q_M\subset G$ with $|G\setminus\cup_{j=1}^M Q_j|<\epsilon|G|$ satisfying the following:
\begin{itemize}
\item[(a)] $0<\mathrm{rad}(Q_j)<\epsilon\;\;\forall\,1\le j\le M;$
\item[(b)] $ \dv \tilde V=0$ in $G$  and  $(D\tilde u,  \tilde V)\in \Sigma^{s}(\mu) $ on  $\bar G;$
\item[(c)]  $\|\tilde u-u\|_{L^\infty(G)} <\epsilon;$
\item[(d)] $|Q_j\cap\{(D\tilde u, \tilde V) \in S^{s}(\mu) \}| \ge (1-\epsilon)|Q_j|\;\; \forall\, 1\le j\le M;$
\item[(e)] $|\{(D\tilde u, \tilde V) \in S^{s}(\mu) \}| \ge (1-\epsilon)|G|;$
\item[(f)]  $\begin{cases}
  |\{(D\tilde u,\tilde V) \in S_k^{s}(\mu) \}| \ge \frac12 (\mu- \lambda) (\mu-\delta_2)^{N-1}\delta_1 |G|,\\
 |\{(D\tilde u,\tilde V) \in S_k^{s}(\mu) \}| \ge \frac{\lambda}{\mu} |\{(D u,V) \in  S_k^r(\lambda) \}| \end{cases}\forall\, 1\le k\le N;$
 \item[(g)] $\|D\tilde u-Du\|_{L^1(G)}\le C_0\big[ |F_0|+   (\epsilon+(\mu- \lambda) ) |G|\big ],$ where
 \[
 F_0=\{(D u,V) \notin S^{r}(\lambda)\}
 \] and  $C_0>0$  is a constant depending only on  $\delta_2$ and the  set $\Sigma(1).$
\end{itemize}
  \end{theorem}

\begin{proof} We carry out the proof in several steps.

\medskip

{\em Step 1.} By Remark \ref{remk33}(iii), let  
\[
d'=\min\Big\{ \dist(\Sigma^{r}(\mu),\partial\Sigma^{s}(\mu)), \quad \min\limits_{1\le k\le N}\dist(S^{r}_k(\mu),\partial S_k^{s}(\mu))\Big\} >0.
\]

Define
\[
\begin{cases} G_k= \{x\in G: (D u(x),V(x)) \in S_k^{r}(\lambda)\}, \quad 1\le k\le N,\\
 I=\{1\le k\le N:\;\; G_k\ne \emptyset\}. \end{cases}
\]
By (P2) of Condition $O_N$,  the sets $\{G_k\}_{k\in I}$  are nonempty, open,  and pairwise disjoint.

For each $k\in I,$ let $G_k'\subset G_k$ be a nonempty open set such that
\begin{equation}\label{set-G'}
  |\partial G_k'|=0, \quad |G_k\setminus G_k'|\le \epsilon'\,|G_k|,
\end{equation}
where  $0<\epsilon'<1$ is  to be determined  later. Define
\[
G'_0= G\setminus \overline{\cup_{k\in I}  G_k'}= G\setminus \cup_{k\in I}  \overline{G_k'}.
\]
Then   the sets $\{G'_k\}_{k\in I\cup \{0\}}$ are  open and pairwise disjoint,  and
\begin{equation}\label{null-G'}
 |G\setminus\cup_{k\in I\cup \{0\}} G_k'|=0.
\end{equation}

\medskip

{\em Step 2.}  Let $\bar y\in G;$ then $Y=Y_{\bar y} =(Du (\bar y), V(\bar y)) \in \Sigma^{r}(\lambda),$ so 
\[
Y\in L_i^{r}(\lambda')
\]
 for some (perhaps not unique)  $1\le i=i_{\bar y} \le N$ and $\delta_2\le\lambda'=\lambda'_{\bar y}\le \lambda.$ Thus,
\begin{equation}\label{Y-form}
Y=q\zeta_{i}(\lambda',\rho) +(1-q)\pi_{i}(\rho')
\end{equation}
for some  $0\le q=q_{\bar y} \le 1$ and $\rho=\rho_{\bar y}, \, \rho'=\rho'_{\bar y} \in \B_{r}$  with $\zeta_{i}(\lambda',\rho) -\pi_{i}(\rho')\in \Gamma.$

In particular, if $\bar y\in G_k'$ for some  $k\in I,$  then $Y \in S^{r}_{k}(\lambda)$; thus, $Y=\zeta_k(\lambda,\rho)$ for some $\rho\in\B_r.$  In this case, we set $q_{\bar y}=1,$  $i_{\bar y}=k,$ $\lambda'_{\bar y}=\lambda$ and $ \rho_{\bar y}= \rho'_{\bar y} =\rho.$

Let  $X_j=\zeta_j(\mu,\rho_{\bar y})$ and $X'_j=\zeta_j(\mu,\rho'_{\bar y})$ for all $1\le j\le N.$

Applying Theorem \ref{lem1}  with $Y=Y_{\bar y}$ and  $\tau=\epsilon'$  on  the cube $Q=Q_{\bar y,l},$  where $0<l<\ell_{\bar y}:=\sup\{l>0 : Q_{\bar y,l}\subset G\},$  we obtain   $(\varphi, \Psi)=(\varphi_{\bar{y},l}, \Psi_{\bar{y},l})\in C^\infty_0(Q;\R^m\times\R^{m\times n})$  such that
\begin{itemize}
\item[(i)] $\dv\Psi=0$  and $Y+(D\varphi,  \Psi)\in \Sigma^{r}(\mu)$ on $\bar Q;$
\item[(ii)]    $\|\varphi\|_{L^\infty(Q)} <\epsilon';$
\item[(iii)]  there exist pairwise disjoint open  sets $P'_1,\dots, P'_N,   P''_1,\dots, P''_N\subset\subset Q$ such that
\begin{equation}\label{eq-d2}
\begin{cases}
 Y+(D\varphi,  \Psi)=\chi_{P'_j} X_j+\chi_{P''_j} X_j' \quad \text{in $P_j=P'_j\cup P''_j$} & \forall\, 1\le j\le N,\\
|P_i|\ge  (1-\epsilon') \big [(1-q) \nu_i^i(\mu,\rho') +q  \big (\frac{\lambda'}{\mu} +(1-\frac{\lambda'}{\mu})  \nu_i^i(\mu,\rho)\big )\big  ] |Q|,\\
|P_j| \ge (1-\epsilon')  \big [(1-q) \nu_i^j(\mu,\rho') +q  (1-\frac{\lambda'}{\mu})  \nu_i^j(\mu,\rho) \big ] |Q| &\forall\, 1\le j\le N.
\end{cases}
\end{equation}
In particular,  $\sum_{j=1}^N |P_j|\ge  (1-\epsilon')|Q|.$
\end{itemize}

By (i) and (iii), we have
\begin{equation}\label{dist-1}
\begin{cases} \dist(Y+(D\varphi,  \Psi),\partial\Sigma^{s}(\mu))\ge d' &\mbox{on $\bar Q,$}  \\
\dist (Y+(D\varphi,  \Psi),\partial S_j^{s}(\mu))\ge d' &\mbox{in $P_j$} \quad \forall\, 1\le j\le N.
\end{cases}
\end{equation}
Moreover, since $\delta_2\le\lambda'\le\lambda$ and $1-\frac{q\lambda'}{\mu} \ge \mu-\lambda,$ it follows from (\ref{eq-1}) that
\begin{equation}\label{eq-e}
\begin{split}
|P_i| &\ge  (1-\epsilon') \Big [ \frac{q\lambda'}{\mu} +\Big (1-\frac{q\lambda'}{\mu}\Big )  (\mu-\delta_2)^{N-1}\delta_1 \Big  ] |Q|,   \\
|P_j|
& \ge (1-\epsilon')  (\mu-\lambda)  (\mu-\delta_2)^{N-1}\delta_1 |Q| \qquad \forall\, 1\le j  \le N.
\end{split}
\end{equation}

Define  $\tilde u  =u_{\bar y,l}=  u + \varphi$ and $\tilde  V  =V_{\bar y,l}=  V +  \Psi$ on $\bar Q=\bar Q_{\bar y,l}.$  Then 
\begin{equation}\label{new-fun0}
\begin{cases}
(\tilde u,\tilde V) \in  (u,V) +C_0^\infty(Q;\R^m\times\R^{m\times n}),\\
  \mbox{$\dv \tilde  V$  in $Q,$}\\
  \|\tilde u-u\|_{L^\infty(Q)} < \epsilon'.
  \end{cases}
\end{equation}

By the uniform continuity of $(Du,V)$ on $\bar G,$  we select  $\ell'>0$ such that
\[
 |(D u(y),V(y))-(D u(y'),V(y'))|<\min\Big\{\frac{d'}{4},\epsilon'\Big\} \quad \forall\, y,y'\in \bar G,\;\;
 |y-y'|<\ell'.
 \]
Let $0<l<\min\{\ell_{\bar y},\ell'\}$.  Then, for all $y\in \bar Q=\bar Q_{\bar y,l},$
\[
  |(D\tilde  u(y),   \tilde V(y)) - [Y+(D\varphi(y),  \Psi(y))] |  = |(D u(y),V(y))-(D u(\bar y),V(\bar y))|< \min\Big\{\frac{d'}{4},\epsilon'\Big\}. 
\]

Let $\epsilon'>0$ be such that
\begin{equation}\label{choice-e'-1}
\epsilon'<\min\Big\{\frac{d'}{4},\frac{\epsilon}{2}\Big\}.
\end{equation}
Then, by  (i),  (iii) and (\ref{dist-1}), for all $0<l<\min\{\ell_{\bar y},\ell'\}$ with $Q=Q_{\bar{y},l}$, we have
\begin{equation}\label{cubes}
 \begin{cases}
 (D u, V)\big |_{\bar Q} \in \B_\epsilon(Y), \\
  (D\tilde  u,  \tilde  V)\big |_{\bar Q} \in \Sigma^{s}(\mu),  \\
  (D\tilde  u,  \tilde  V)\big |_{P_j} \in  S_j^{s}(\mu)\cap [\B_\epsilon (X_j) \cup \B_\epsilon (X'_j)]   \quad  \forall\, 1 \le j \le N,\\
  | Q\cap \{(D\tilde u, \tilde V)\in S^{s}(\mu)\}|\ge |\cup_{j=1}^N P_j|\ge (1-\epsilon')|Q|.
  \end{cases}
\end{equation}

\medskip

{\em Step 3.}   Let $k\in I_0:=I\cup\{0\}$. Choose a sequence $\{Q^k_\nu\}_{\nu=1}^\infty$ of disjoint cubes in $G_k'$ with $0<\mathrm{rad}(Q^k_\nu)<\min\{\ell',\epsilon\}$ for all $\nu\ge1$ such that
\[
\big|G_k'\setminus\cup_{\nu=1}^\infty Q^k_\nu\big|=0.
\]
Pick a large integer $m_k\ge1$ so that
\begin{equation}\label{choice-Qs}
\big|G_k'\setminus\cup_{\nu=1}^{m_k} Q^k_\nu\big|\le\epsilon'|G_k'|.
\end{equation}
For each  $\nu\ge 1,$ let $\bar y^k_\nu$ denote the center of $Q^k_\nu$ and $l^k_\nu=\mathrm{rad}(Q^k_\nu)$.

Following the constructions in Step 2, we define
\begin{equation}\label{final-fun}
(\tilde  u,\tilde V)  =    \sum_{k\in I_0} \sum_{\nu=1}^{m_k} (u_{\bar{y}^k_\nu,l^k_\nu}, V_{\bar{y}^k_\nu,l^k_\nu}) \chi_{ Q^k_\nu}+ (u,V)\chi_{G\setminus G'}\quad\mbox{in $G$},
\end{equation}
where $G'=\bigcup_{k\in I_0}\bigcup_{\nu=1}^{m_k}  Q^k_\nu.$
Then $(\tilde u,\tilde V) \in  (u,V) +C_0^\infty(G;\R^m\times\R^{m\times n})$ and
\[
|G\setminus G'|=\sum_{k\in I_0}|G_k'\setminus\cup_{\nu=1}^{m_k} Q^k_\nu|\le\epsilon'|G|<\epsilon|G|.
\]

We relabel the cubes $Q^k_\nu$ $(k\in I_0,\,1\le\nu\le m_k)$ as $Q_1,\ldots,Q_M.$ Then requirement (a) is satisfied.
From (\ref{strict-sub}), (\ref{new-fun0}), (\ref{choice-e'-1}), (\ref{cubes}) and (\ref{final-fun}), requirements (b) and (c) also follow.

Moreover, using (\ref{null-G'}), (\ref{choice-e'-1}), (\ref{cubes}) and (\ref{choice-Qs}), we have
\[
\begin{split}
 |\{(D\tilde u, \tilde V) &\in S^{s}(\mu)\}|  = \sum_{k\in I_0}|G_k' \cap\{(D\tilde u, \tilde V)\in S^{s}(\mu)\}|\\
& = \sum_{k\in I_0}\sum_{\nu=1}^\infty|Q^k_\nu \cap\{(D\tilde u, \tilde V)\in S^{s}(\mu)\}|  \ge \sum_{k\in I_0}\sum_{\nu=1}^{m_k}|Q^k_\nu \cap\{(D\tilde u, \tilde V)\in S^{s}(\mu)\}|\\
& \ge(1-\epsilon')\sum_{k\in I_0}\sum_{\nu=1}^{m_k}| Q^k_\nu|  =(1-\epsilon')\sum_{k\in I_0}\big(|G_k'|- |G'_k\setminus \cup_{\nu=1}^{m_k} Q^k_\nu|\big)\\
& \ge (1-\epsilon')^2 \sum_{k\in I_0}|G_k'|=(1-\epsilon')^2|G|>(1-\epsilon)|G|.
\end{split}
\]
Hence, requirements (d) and (e) are satisfied.

\medskip

{\em Step 4.}  We  now verify  requirement  (f).
First,  from the second  of  (\ref{eq-e}) and third of (\ref{cubes}), we have that  for each $1\le j \le N,$
\[
 \begin{split}
  |\{(D\tilde u, \tilde V)   \in S_j^{s}(\mu)\}|  & \ge \sum_{k\in I_0}\sum_{\nu=1}^{m_k}|Q^k_\nu \cap\{(D\tilde u, \tilde V)\in S_j^{s}(\mu)\}|  \\
  & \ge (1-\epsilon')  (\mu-\lambda)(\mu-\delta_2)^{N-1}\delta_1 \sum_{k\in I_0} \sum_{\nu=1}^{m_k} |Q_\nu^k| \\
  & \ge  (1-\epsilon')^2    (\mu-\lambda)(\mu-\delta_2)^{N-1}\delta_1  |G|  \\ &  \ge  \frac12 (\mu-\lambda)(\mu-\delta_2)^{N-1}\delta_1  |G|, \end{split}
\]
where $\epsilon'\in(0,1)$ is chosen to satisfy, in addition to (\ref{choice-e'-1}),
\begin{equation}\label{choice-e'-2}
 (1-\epsilon')^2\ge 1/2.
\end{equation}
This verifies  the first inequality of (f).

Next, let $1\le k\le N.$ If $k\notin I,$ then 
\[
G_k=\{y\in G: (Du(y),V(y))\in S^r_k(\lambda)\}=\emptyset.
\]
 Thus,  the second inequality of (f)  is automatically satisfied.
Now assume $k\in I;$ then
\[
\bar{y}^k_\nu\in Q^k_\nu\subset G_k'\subset G_k\ne\emptyset \qquad (1\le \nu\le m_k).
\]
Hence,  in  \eqref{Y-form} of  Step 2, we have
\[
 q=1,\quad i=k, \quad \lambda'=\lambda, \quad  \rho=\rho'\in\B_r.
\]

 From (\ref{set-G'}),  the first of \eqref{eq-e},  the third of (\ref{cubes}),  and (\ref{choice-Qs}), it follows that
 \[
 \begin{split} |\{(D\tilde u, \tilde V)\in S_k^{s}(\mu)\}|
 &  \ge   \sum_{j\in I_0}\sum_{\nu=1}^{m_j}  | Q_\nu^j \cap \{(D  \tilde u,  \tilde V)\in S_k^{s}(\mu)\}|   \\
 & \ge   \sum_{\nu=1}^{m_k}  | Q_\nu^k \cap \{(D  u_{\bar{y}^k_\nu,l^k_\nu}, V_{\bar{y}^k_\nu,l^k_\nu})\in S_k^{s}(\mu)\}|  \\
 &  \ge  (1-\epsilon') \Big [\frac{\lambda}{\mu} +\Big (1-\frac{\lambda}{\mu}\Big )(\mu-\delta_2)^{N-1} \delta_1\Big ] \sum_{\nu=1}^{m_k} |Q_\nu^k|\\
 &   \ge (1-\epsilon')^2 \Big [ \frac{\lambda}{\mu} +\Big (1-\frac{\lambda}{\mu}\Big )(\mu-\delta_2)^{N-1}\delta_1 \Big ] |G_k'|\\
 &  \ge (1-\epsilon')^3 \Big  [ \frac{\lambda}{\mu} +\Big (1-\frac{\lambda}{\mu}\Big )(\mu-\delta_2)^{N-1}\delta_1 \Big ]  |G_k|. \end{split}
 \]
Therefore,  the second inequality of (f) is ensured  if   $\epsilon'\in (0,1)$ is chosen to further satisfy, along with (\ref{choice-e'-1}) and (\ref{choice-e'-2}), that
 \begin{equation}\label{small1}
 (1-\epsilon')^3 \Big [ \frac{\lambda}{\mu} +\Big (1-\frac{\lambda}{\mu}\Big )(\mu-\delta_2)^{N-1}\delta_1 \Big ]\ge \frac{\lambda}{\mu},
 \end{equation}
   which is possible since  $\big (1-\frac{\lambda}{\mu}\big )(\mu-\delta_2)^{N-1}\delta_1>0.$

\medskip

{\em Step 5.}  Finally, we verify  requirement (g).
 In the following, we denote by $C$ any constant  depending  on  $\delta_2$ and the diameter of the set $\Sigma(1).$
 
Note from (\ref{set-G'}), (\ref{choice-e'-1}) and (\ref{choice-Qs})  that
 \begin{equation}\label{f0}
 \begin{split}
  \|D\tilde u -Du\|_{L^1(G)}  = &\,  \|D\tilde u-Du\|_{L^1(F_0)} + \sum_{k\in I}  \|D\tilde u-Du\|_{L^1(G_k\setminus G_k')}
\\
& \, + \sum_{k\in I}  \|D\tilde u-Du\|_{L^1(G_k'\setminus \cup_{\nu=1}^{m_k}Q^k_\nu)} + \sum_{k\in I}\sum_{\nu=1}^{m_k}  \|D\tilde u-Du\|_{L^1(Q^k_\nu)}\\
\le &\, C(|F_0| + \epsilon'|G|)  + \sum_{k\in I}\sum_{\nu=1}^{m_k} \int_{Q^k_\nu} |(D\tilde u,\tilde V)-(D u,V)|
 \\
 \le &\, C(|F_0| + \epsilon |G|)  + \sum_{k\in I} \sum_{\nu=1}^{m_k}\int_{ P^k_{\nu}}|(D\tilde u, \tilde V) -(D u,V)|
 \\
&  \, + \sum_{k\in I}\sum_{\nu=1}^{m_k} \int_{Q^k_\nu\setminus  P^k_{\nu}}|(D\tilde u, \tilde V)-(D u, V)|, \end{split}
\end{equation}
where  $P^k_{\nu}\subset\subset Q_\nu^k=Q_{\bar y^k_\nu,l^k_\nu}$ ($k\in I$,  $1\le \nu\le m_k$) are the sets defined as in Step 2, with
\[
Y= Y^k_\nu= Y_{\bar{y}^k_\nu}= (Du(\bar{y}^k_\nu),V(\bar{y}^k_\nu)) \in S_k^r(\lambda), \quad Q=Q^k_\nu, \quad 
P_i=P^k_\nu.
\]
Hence,  in \eqref{Y-form}, we have   $q=1$, $i=k$,  $\lambda'=\lambda$ and $\rho=\rho'=\rho_{\bar{y}^k_\nu}\in \B_r.$

By (\ref{small1}), we have
\[
  (1-\epsilon')  \Big [ \frac{\lambda}{\mu} +\Big (1-\frac{\lambda}{\mu}\Big ) (\mu-\delta_2)^{N-1}\delta_1 \Big ]  >  \frac{\lambda}{\mu},
\]
which, by  the first of \eqref{eq-e},  implies that $| P^k_{\nu}|>\frac{\lambda}{\mu}| Q_\nu^k|.$  Hence,
\[
 |Q^k_\nu\setminus P_\nu^k| <\frac{\mu-\lambda}{\mu}|Q^k_\nu|< \frac{\mu-\lambda}{\delta_2}|Q^k_\nu|,
\]
so  
\begin{equation}\label{f2}
  \int_{Q^k_\nu\setminus  P^k_{\nu}}  |(D\tilde u,  \tilde V) -(D u, V)|  \le C|Q^k_\nu\setminus  P_\nu^k|\le C(\mu-\lambda)|Q^k_\nu|.
\end{equation}
Moreover, let  
\[
X^k_{\nu}=\zeta_k(\mu,\rho) \in S_k^{r}(\mu)
\]
be defined as in Step 2 corresponding to  $Y=Y^k_\nu$  with $\rho=\rho_{\bar{y}^k_\nu}\in \B_r.$  Then
\[
|X_\nu^k-Y^k_\nu|=(\mu-\lambda)|\xi_k(\rho_{\bar{y}^k_\nu})-\pi_k(\rho_{\bar{y}^k_\nu})|\le C(\mu-\lambda).
\]
Thus, by (\ref{cubes}),  we have
\begin{equation}\label{f1}
 \begin{split}
 \int_{ P_\nu^k}  & |(D\tilde u,  \tilde V)  -(D u,V)| \,dx \\
 &  \le  \int_{ P^k_\nu} \big (|(D\tilde u, \tilde V)-X_\nu^k|+|X_\nu^k-Y^k_\nu|+|(D u,V)-Y^k_\nu|\big)\\
 & \le C(\epsilon+(\mu-\lambda))|Q^k_\nu|. \end{split}
  \end{equation}
Inserting (\ref{f2}) and (\ref{f1}) into (\ref{f0}) proves  requirement (g).

\medskip

This completes the proof.
\end{proof}

%%%%%%%%%%%%%%%

\subsection{Constructions of $\U_\nu$ and $(u_\nu,V_\nu)$}

Let   $N\ge 2$ and let 
\[
\sigma\colon\R^{m\times n}\to \R^{m\times n}
\]
  be locally Lipschitz  and satisfy Condition $O_N.$  
  Assume $u\in C^1(\bar \Omega;\R^m)$ and  $V \in C(\bar \Omega; \R^{m\times n})$ satisfy 
\[
 \dv V=0 \; \text{ in $\Omega$}, \quad (Du, V)\in   \Sigma(1) \; \text{ on  $\bar \Omega.$}
\]

Since the set 
\[
\Delta=(D\bar u, \bar V)(\bar\Omega )\subset \Sigma(1)
\]
 is compact, by  Remark  \ref{remk33}(iii), we have
 \begin{equation}\label{subs-2}
\Delta=(D u_0,V_0)(\bar\Omega ) \subset  \Sigma^{r_1}(\lambda_1),
\end{equation}
for some  $ r_1 \in (0, r_0)$ and  $\lambda_1 \in [ \delta_2 , 1)$,   
where 
\[
(u_0,V_0):=(\bar u,\bar V) \; \text{ in $\bar\Omega.$}
\]

Let $0<\delta<1.$ Define the sequences $\{\lambda_\nu\}_{\nu=2}^\infty$, $\{r_\nu\}_{\nu=2}^\infty$, and $\{\epsilon_\nu\}_{\nu=1}^\infty$ as follows:
\begin{equation}\label{select-1}
\lambda_{\nu+1} = \frac{1}{2} (1 + \lambda_\nu),\quad r_{\nu+1} = \frac{1}{2} (r_0 + r_\nu),\quad  \epsilon_\nu = \frac{\delta}{3^\nu}
\qquad \forall\, \nu = 1,2,\ldots.
\end{equation}
Then
\[
\delta_2 \le \lambda_1 < \lambda_2 < \cdots < 1, 
\quad 
0 < r_1 < r_2 < \cdots < r_0,
\]
and
\[
\lim_{\nu \to \infty} \lambda_\nu = 1,
\quad
\lim_{\nu \to \infty} r_\nu = r_0.
\]

We now define the open sets
\begin{equation}\label{set-U}
\mathcal{U}_\nu = \Sigma^{r_{\nu+1}}\!\bigl(\lambda_{\nu+1}\bigr), \quad \forall\, \nu \ge 1,
\end{equation}
and proceed to construct the  corresponding sequence 
\[
\{(u_\nu,V_\nu)\}_{\nu=1}^\infty \subset (\bar u, \bar V) + C^\infty_0(\Omega; \R^m \times \R^{m \times n}),
\]
 required in \eqref{pdr2},  inductively as follows.  

\medskip

We begin with \eqref{subs-2}. Applying the main stage Theorem~\ref{thm1} to $(u,V)=(u_0,V_0)$ on the set $G=\Omega$, with
\[
\lambda=\lambda_1,\quad 
\mu=\lambda_2,\quad  
r=r_1,\quad 
s=r_2,\quad  
\epsilon=\epsilon_1,
\]
we obtain a function $(u_1,V_1)=(\tilde u,\tilde V)\in (u_0,V_0)+C_0^\infty(\Omega;\R^m\times \R^{m\times n})$ and finitely many pairwise disjoint cubes $Q^1_1,\ldots,Q^1_{M_1}\subset \Omega$ such that
\[
\bigl|\Omega \setminus \textstyle\bigcup_{j=1}^{M_1} Q^1_j\bigr| < \epsilon_1 |\Omega|,
\]
and the following properties hold:
\begin{equation}\label{u1}
\begin{cases}
0<\rad(Q^1_j)<\epsilon_1\;\;\forall\, 1\le j\le M_1;  \\
\mbox{$\dv V_1=0$ in $\Omega$, \; $(D u_1,  V_1)\in \U_1$  on $\bar \Omega $;}\\
 \|u_1-u_0\|_{L^\infty(\Omega )}  <\epsilon_1;  \\
|Q^1_j\cap \{(D u_1,V_1) \in S^{r_2}(\lambda_2) \}| \ge (1-\epsilon_1)|Q^1_j| \quad \forall\, 1\le j\le M_1;\\
|\{(D u_1,V_1) \in S^{r_2}(\lambda_2)\}| \ge (1-\epsilon_1)|\Omega |; \\
\mbox{$
\begin{cases}
 |\{(D u_1,V_1) \in S_k^{r_2}(\lambda_2)\}| \ge \frac12 (\lambda_2-\lambda_1)(\lambda_2-\delta_2)^{N-1}\delta_1|\Omega |, \\
|\{(D u_1,V_1) \in S_k^{r_2}(\lambda_2) \}| \ge \frac{\lambda_1}{\lambda_2} |\{(D u_0,V_0) \in S_k^{r_1}(\lambda_1) \}|\;\; \forall\, 1\le k\le N;
\end{cases} $} \\
\|Du_1 - Du_0\|_{L^1(\Omega )}\le C_0\big[|F_0|+(\epsilon_1+(\lambda_2-\lambda_1))|\Omega |\big],\\
\mbox{where $F_0=\{(D u_0,V_0) \not\in S^{r_1}(\lambda_1) \}.$}   
\end{cases}
\end{equation}
Here, we assume further that
\begin{equation}\label{u1-0}
C_0\ge\mathrm{diam}(\Sigma(1)).
\end{equation}

Let $Q^1_0=\Omega \setminus\cup_{j=1}^{M_1}\bar{Q}^1_j.$ Then $Q^1_0$ is an open subset of $\Omega,$ with
\begin{equation}\label{u1-1}
 |Q^1_0|<\epsilon_1|\Omega |, \quad |\Omega \setminus\cup_{j=0}^{M_1}Q^1_j|=0.
\end{equation}
 
 For each $0\le i \le M_1,$ we  apply Theorem \ref{thm1} to  $(u,V)=(u_1,V_1)$ on the set $G=Q^1_i$, with
\[
 \lambda=\lambda_2, \quad  \mu=\lambda_3,\quad r=r_2,\quad s=r_3, \quad  \epsilon=\epsilon_2,
 \]
to obtain a function 
\[
(u_i^2,V_i^2)= (\tilde u,\tilde v)\in (u_1,V_1)+C^{\infty}_0(Q^1_i;\R^m\times \R^{m\times n})
\]
 and finitely many disjoint cubes 
 \[
 Q^2_{i,1},\ldots,Q^2_{i,M^2_i}\subset Q^1_i, \; \text{ with } \; |Q^1_i\setminus\cup_{j=1}^{M^2_i}Q^2_{i,j}|<\epsilon_2|Q^1_i|,
 \]
such that
\begin{equation}\label{u2}
\begin{cases}
0<\rad(Q^2_{i,j})<\epsilon_2\;\;\forall\,1\le j\le M^2_i;  \\
\mbox{$ \dv V^2_i=0$ in $Q^1_i$, \; $(D u^2_i,  V^2_i)\in \U_2$  on $\bar Q^1_i$;}\\
\|u^2_i-u_1\|_{L^\infty(Q^1_i)}  <\epsilon_2; \\
|Q^2_{i,j}\cap \{(D u^2_i,V^2_i) \in S^{r_3}(\lambda_3) \}| \ge (1-\epsilon_2)|Q^2_{i,j}| \quad \forall\, 1\le j\le M^2_i;\\
|Q^1_i\cap\{(D u^2_i,V^2_i) \in S^{r_3}(\lambda_3) \}| \ge (1-\epsilon_2)|Q^1_i|; \\
\begin{cases}
 |Q^1_i\cap\{(D u^2_i,V^2_i) \in S_k^{r_3}(\lambda_3)\}| \\ \quad \ge \frac{1}{2}(\lambda_3-\lambda_2)(\lambda_3-\delta_2)^{N-1}\delta_1|Q^1_i|, \\
|Q^1_i\cap\{(D u^2_i,V^2_i) \in S_k^{r_3}(\lambda_3) \}| \\ \quad \ge \frac{\lambda_2}{\lambda_3} |Q^1_i\cap\{(D u_1,V_1) \in S_k^{r_2}(\lambda_2) \}|
\end{cases}   (1\le k\le N);\\
\|Du^2_i - D u_1\|_{L^1(Q^1_i)}\le C_0\big[|F^1_i|+(\epsilon_2+(\lambda_3-\lambda_2))|Q^1_i|\big], \\
 \mbox{where $F^1_i=Q^1_i\cap\{(D u_1,V_1) \not\in S^{r_2}(\lambda_2) \}.$ }\\
\end{cases}
\end{equation}

By the fourth property in \eqref{u1}, we have
\[
|F^1_i| < \epsilon_1 |Q^1_i| \quad \text{for } 1 \le i \le M_1.
\]
If $i=0$, then by \eqref{u1-1} we obtain
\begin{equation}\label{u2-1}
|F^1_0| \le |Q^1_0| < \epsilon_1 |\Omega|.
\end{equation}

Let $Q^{2}_{i,0}=Q^1_i\setminus\cup_{j=1}^{M^2_i}\bar{Q}^2_{i,j}$; then $Q^2_{i,0}$ is an open subset of $Q^1_i,$ with
\[
 |Q^{2}_{i,0}|<\epsilon_2|Q^1_i|, \quad |Q^1_i\setminus\cup_{j=0}^{M^2_i} Q^2_{i,j}|=0.
\]

We  relabel the  open sets  $Q^2_{i,j}$ $(0\le i\le M_1,\,0\le j\le M^2_i)$ as  $Q^2_0,Q^2_1,\ldots,Q^2_{M_2}$ by setting
\[
M_2=\sum_{i=0}^{M_1}(M^2_i+1)-1, \quad Q^2_{\mu}=Q^2_{i,j},  \; \text{ where } \; 
 \mu=i+j+\sum_{k=1}^i M^2_{k-1}, 
\]
for all $0\le i\le M_1$ and $0\le j\le M^2_i.$ Here we adopt the convention that $\sum_{k=1}^0 M^2_{k-1}=0.$
Moreover,  $\big|\Omega \setminus\bigcup_{j=0}^{M_2}Q^2_j\big|=0,$ and for all $0\le i\le M_1,$
\[  
\bigg|Q^1_i\setminus\bigcup_{j=0}^{M^2_i}Q^2_{i+j+\sum_{k=1}^i M^2_{k-1}}\bigg|=0,\quad \big|Q^2_{i+\sum_{k=1}^i M^2_{k-1}}\big|<\epsilon_2|Q^1_i|.
 \]

Define
\[
(u_2,V_2)=\sum_{i=0}^{M_1}(u^2_i,V^2_i)\chi_{Q^1_i}\in (u_0,v_0)+C^{\infty}_0(\Omega ;\R^m\times \R^{m\times n}).
\]
Then
\[
\begin{cases}
0<\rad(Q^2_{j})<\epsilon_2\;\;\mbox{$\forall\, 0\le j\le M_2$ with $j\ne i+\sum_{k=1}^i M^2_{k-1}$ $(0\le i\le M_1)$;}  \\
\mbox{$ \dv V_2=0$ in $\Omega$,  \;  $(D u_2,  V_2)\in \U_2$  on $\bar \Omega $;}\\
\|u_2-u_1\|_{L^\infty(\Omega )}  <\epsilon_2; \\
|Q^2_{j}\cap \{(D u_2,V_2) \in S^{r_3}(\lambda_3) \}| \ge (1-\epsilon_2)|Q^2_{j}| \\
\quad \forall\, 0\le j\le M_2 \;\;  \mbox{with $j\ne i+\sum_{k=1}^i M^2_{k-1}$ $(0\le i\le M_1)$;}\\
|\{(D u_2,V_2) \in S^{r_3}(\lambda_3) \}| \ge (1-\epsilon_2)|\Omega |; \\
\begin{cases}
 |Q^1_i\cap\{(D u_2,V_2) \in S_k^{r_3}(\lambda_3) \}| \ge \frac{1}{2}(\lambda_3-\lambda_2)(\lambda_3-\delta_2)^{N-1}\delta_1|Q^1_i|, \\
|Q^1_i\cap\{(D u_2,V_2) \in S_k^{r_3}(\lambda_3) \}| \ge \frac{\lambda_2}{\lambda_3} |Q^1_i\cap\{(D u_1,V_1) \in S_k^{r_2}(\lambda_2) \}|,\\
\mbox{for each $1\le k\le N$ and $0\le i\le M_1$};
\end{cases} \\
\|Du_2 - D u_1\|_{L^1(\Omega )}\le  C_0\big[2\epsilon_1+\epsilon_2+(\lambda_3-\lambda_2)\big]|\Omega |.
\end{cases}
\]
Here, the last inequality follows from  (\ref{u1-0}), (\ref{u2}) and (\ref{u2-1}), since 
\[
\begin{split}
\|Du_2 - D u_1   \|_{L^1(\Omega )}&=\sum_{j=0}^{M_1} \|Du_2 - D u_1\|_{L^1(Q^1_j)}\\
&= \|Du_2 - D u_1\|_{L^1(Q^1_0)} +\sum_{j=1}^{M_1} \|Du_2 - D u_1\|_{L^1(Q^1_j)}\\
& \le C_0 |Q^1_0|+\sum_{j=1}^{M_1} C_0 \big[|F^1_j|+(\epsilon_2+(\lambda_3-\lambda_2))|Q^1_j|\big] \\
& \le C_0\epsilon_1|\Omega |+C_0[ \epsilon_1+\epsilon_2+(\lambda_3-\lambda_2)]  \sum_{j=1}^{M_1}  |Q^1_j| \\
& \le  C_0 \big[2\epsilon_1+\epsilon_2+(\lambda_3-\lambda_2)\big]|\Omega |.
\end{split}
\]

Repeating this process indefinitely, we obtain a sequence 
\[
\{(u_\nu, V_\nu)\}_{\nu=1}^\infty \subset 
(u_0,V_0)+C^{\infty}_0(\Omega ;\R^m\times \R^{m\times n})
\]
and finitely many disjoint open sets $Q^\nu_0,Q^\nu_1,\ldots,Q^\nu_{M_\nu}\subset\Omega $ with 
\[
|\Omega \setminus\cup_{j=0}^{M_\nu}Q^\nu_j|=0\quad \forall\, \nu=1,2,\ldots, 
\]
satisfying that,
for each $\nu\ge2$,
\[
M_\nu=\sum_{i=0}^{M_{\nu-1}}(M^\nu_i +1)-1,
\]
that all  sets $Q^\nu_j$, except those with  $j =i+\sum_{k=1}^i M^\nu_{k-1}$ for $0\le i\le M_{\nu-1}$, are  cubes in $\R^{n}$, 
 and that, for each $\nu\ge2$ and $0\le i\le M_{\nu-1},$
\begin{equation}\label{u-nu-0}
\Big |Q^{\nu-1}_i \setminus\bigcup_{j=0}^{M^\nu_i} Q^{\nu}_{i+j+\sum_{k=1}^i M^\nu_{k-1}} \Big |=0, \quad \Big |Q^{\nu}_{i+\sum_{k=1}^i M^\nu_{k-1}}\Big |<\epsilon_\nu|Q^{\nu-1}_i|.
\end{equation}

Moreover, for each $\nu\ge2$,
\begin{equation}\label{u-nu}
\begin{split}
& 0<\rad(Q^{\nu}_{j})<\epsilon_{\nu}\;\;\mbox{$\forall\, 0\le j\le M_{\nu}$,  $j\ne i+\sum_{k=1}^i M^\nu_{k-1}$ $(0\le i\le M_{\nu-1})$;}  
\\ &
 \mbox{$\dv V_{\nu}=0$ in $\Omega$, \;  $(D u_{\nu},  V_{\nu})\in \U_\nu$  on $\bar \Omega $;}\\ &
 \|u_{\nu}-u_{\nu-1}\|_{L^\infty(\Omega )}  <\epsilon_{\nu}; \\ &
 |Q^{\nu}_{j}\cap \{(D u_{\nu},V_{\nu}) \in S^{r_{\nu+1}}(\lambda_{\nu+1}) \}| \ge (1-\epsilon_{\nu})|Q^{\nu}_{j}|   \\
 & \mbox{ $\forall\, 0\le j\le M_{\nu},  \;\; j\ne i+\sum_{k=1}^i M^\nu_{k-1} \quad (0\le i\le M_{\nu-1});$}  \\ &
 |\{(D u_{\nu},V_{\nu}) \in S^{r_{\nu+1}}(\lambda_{\nu+1}) \}| \ge (1-\epsilon_{\nu})|\Omega |; \\ &
 \begin{cases}
 |Q^{\nu-1}_i\cap\{(D u_{\nu},V_{\nu}) \in S_k^{r_{\nu+1}}(\lambda_{\nu+1}) \}| \\ \quad \ge \frac{1}{2}(\lambda_{\nu+1}-\lambda_{\nu})(\lambda_{\nu+1}-\delta_2)^{N-1}\delta_1|Q^{\nu-1}_i|, \\
|Q^{\nu-1}_i\cap\{(D u_{\nu},V_{\nu}) \in S_k^{r_{\nu+1}}(\lambda_{\nu+1}) \}| \\  \quad \ge \frac{\lambda_{\nu}}{\lambda_{\nu+1}} |Q^{\nu-1}_i\cap\{(D u_{\nu-1},V_{\nu-1}) \in S_k^{r_{\nu}}(\lambda_{\nu}) \}|,
\end{cases}  \forall\, 1\le k\le N, \;\; 0\le i\le M_{\nu-1};  \\ &
 \|Du_{\nu} - D u_{\nu-1}\|_{L^1(\Omega )}\le  C_0 \big[2\epsilon_{\nu-1}+\epsilon_{\nu}+(\lambda_{\nu+1}-\lambda_{\nu})\big]|\Omega |.
\end{split}
\end{equation}

We have the following result; see also \cite[Lemma 4.2]{GKY26}.

\begin{lemma}\label{lem512} For  all $p>q\ge 1$, $0\le j\le M_q$ and  $1\le k\le N,$  we have
\[
 |Q^q_j  \cap\{( Du_{p} ,V_{p})  \in S_k^{r_{p+1}}(\lambda_{p+1}) \}|  \ge  \frac12\, \frac{\lambda_{q+2}}{\lambda_{p+1}}  (\lambda_{q+2}- \lambda_{q+1})(\lambda_{q+2}-\delta_2)^{N-1}\delta_1\,|Q^q_j|.
\]
\end{lemma}

\begin{proof}
From the construction of  $\{(u_{p},V_{p})\}_{p=1}^\infty$ and  $\{Q^{p}_j\}_{p\ge1,\,0\le j\le M_p}$, we  see that  if $p>q\ge 1$ and $0\le j\le M_q$, then
\[
Q^q_j=(\cup_{k\in I} Q^p_k)\cup E
\]
 for some index set $I\subset\{0,1,\ldots, M_p\}$ and  null-set $E.$
Thus,  for all $p>q\ge 1$, $0\le j\le M_q$ and $1\le k\le N,$   by  (\ref{u-nu}), we have
\[
\begin{split} |Q^q_j  \cap &\{( Du_{p},V_{p})  \in S_k^{r_{p+1}}(\lambda_{p+1}) \}| \\
&  \ge \frac{\lambda_p}{\lambda_{p+1}}\, |Q^q_j\cap\{(Du_{p-1},V_{p-1}) \in S_k^{r_{p}}(\lambda_p) \}|\\
& \ge \frac{\lambda_p}{\lambda_{p+1}} \frac{\lambda_{p-1}}{\lambda_{p}}\,|Q^q_j\cap\{(Du_{p-2},V_{p-2}) \in S_k^{r_{p-1}}(\lambda_{p-1}) \}|\\
&\hspace{30ex}  \vdots \\
& \ge \frac{\lambda_p}{\lambda_{p+1}}  \frac{\lambda_{p-1}}{\lambda_{p}}   \cdots  \frac{\lambda_{q+2}}{\lambda_{q+3}}\, |Q^q_j\cap\{(Du_{q+1},V_{q+1}) \in S_k^{r_{q+2}}(\lambda_{q+2}) \}| \\
& =\frac{\lambda_{q+2}}{\lambda_{p+1}}\, |Q^q_j\cap\{(Du_{q+1},V_{q+1}) \in S_k^{r_{q+2}}(\lambda_{q+2}) \}| \\
& \ge    \frac12 \, \frac{\lambda_{q+2}}{\lambda_{p+1}}  (\lambda_{q+2}- \lambda_{q+1})(\lambda_{q+2}-\delta_2)^{N-1}\delta_1\,|Q^q_j|.
\end{split}
\]
\end{proof}

The following   lemma is the same as \cite[Lemma 4.3]{GKY26}; the proof is therefore omitted.

\begin{lemma} \label{lem52} Let $f(\xi)= \sigma(\xi^1)-\xi^2 $ for $\xi=(\xi^1,\xi^2)\in \R^{m\times n}\times \R^{m\times n}.$
Then
\[
 |f(\xi)|\le C\,(1-\lambda) \quad \forall\, 0\le \lambda\le 1,\;\; \xi \in S^{r_0}(\lambda),
 \]
 where $C>0$ is a constant.
\end{lemma}

\subsection{Proof of Theorem~\ref{mainthm}}

We complete the proof  in several steps. 

\medskip

\noindent
\emph{Step 1.}
By the third property in both \eqref{u1} and \eqref{u-nu}, for each $j \ge 1$,
\begin{equation}\label{norm-0}
\|u_j - u_0\|_{L^\infty(\Omega)}
\le \sum_{\nu=0}^{j-1} \|u_{\nu+1} - u_\nu\|_{L^\infty(\Omega)}
\le \sum_{\nu=0}^{\infty} \epsilon_{\nu+1}
= \frac{\delta}{2}.
\end{equation}
Hence, the sequence $\{u_\nu\}_{\nu=1}^\infty$ is uniformly bounded on $\Omega$. By the second condition in \eqref{u-nu},  the sequence $\{Du_\nu\}_{\nu=1}^\infty$ is  uniformly bounded on $\Omega$. Consequently, $\{u_\nu\}_{\nu=1}^\infty$ is uniformly bounded in $\bar u + W_0^{1,\infty}(\Omega;\R^m).$

Moreover, by the third and last conditions in \eqref{u-nu}, the sequence $\{u_\nu\}_{\nu=1}^\infty$ is Cauchy in $W^{1,1}(\Omega;\R^m)$. Consequently, there exists $u \in W^{1,1}(\Omega;\R^m)$ such that
\[
 \text{$u_\nu\to u$ \, in   $W^{1,1}(\Omega ;\R^m)$ \, as $\nu\to\infty.$}
\]
We also have  that $u\in \bar u+ W^{1,\infty}_{0}(\Omega ;\R^m).$

Passing to the limit in \eqref{norm-0} yields
 \[
 \|u  -\bar u\|_{L^\infty(\Omega )}     \le  \frac{\delta}{2} <\delta.
\]

\emph{Step 2.}
By \eqref{u-nu} and Lemma~\ref{lem52}, we have
\[
\begin{aligned}
\| \sigma(Du_\nu) - V_\nu \|_{L^1(\Omega)}
&= \int_{\Omega} |f(Du_\nu, V_\nu)| \, dx \\
&= \int_{(Du_\nu,V_\nu)\in S^{r_{\nu+1}}(\lambda_{\nu+1})}
|f(Du_\nu, V_\nu)| \, dx \\
&\quad
+ \int_{(Du_\nu,V_\nu)\notin S^{r_{\nu+1}}(\lambda_{\nu+1})}
|f(Du_\nu, V_\nu)| \, dx \\
&\le C\bigl[(1-\lambda_{\nu+1}) + \epsilon_\nu\bigr] |\Omega| \to 0
\quad \text{as } \nu \to \infty.
\end{aligned}
\]
Moreover, since $\operatorname{div} V_\nu = 0$ in $\Omega$ (in the sense of distributions), we obtain
\begin{equation}\label{pre-weak}
\int_{\Omega} \langle \sigma(Du_\nu), D\varphi \rangle \, dx
=
\int_{\Omega} \langle f(Du_\nu, V_\nu), D\varphi \rangle \, dx
\quad
\forall\, \varphi \in C_0^\infty(\Omega;\R^m).
\end{equation}

Since $\{\|Du_\nu\|_{L^\infty(\Omega)}\}$ is uniformly bounded and
\[
\|u_\nu - u\|_{W^{1,1}(\Omega)} \to 0,
\qquad
\|f(Du_\nu, V_\nu)\|_{L^1(\Omega)} \to 0
\quad \text{as } \nu \to \infty,
\]
we may pass to the limit in \eqref{pre-weak} to obtain
\[
\int_{\Omega} \langle \sigma(Du), D\varphi \rangle \, dx
= 0
\quad
\forall\, \varphi \in C_0^\infty(\Omega;\R^m).
\]

\medskip

Consequently, Steps~1 and~2 show that $u$ is a Lipschitz weak solution of the Dirichlet problem \eqref{DP} and satisfies
\[
\|u - \bar u\|_{L^\infty(\Omega)} < \delta,
\]
as required in Theorem~\ref{mainthm}.

\medskip

\emph{Step 3.}    Let $y_0 \in \Omega$, and let $U$ be any open subset of $\Omega$ containing $y_0$.
Choose $l_0>0$ such that $Q_{y_0,2l_0} \subset U$.
Since
\[
|Q_{y_0,l_0}| = \bigl| Q_{y_0,l_0} \cap \bigl( \cup_{j=0}^{M_\nu} Q^\nu_j \bigr) \bigr|
\quad \text{for all } \nu \ge 1,
\]
it follows from \eqref{u-nu-0} that there exists $q_0 \ge 2$ such that, for all $q \ge q_0$,
\[
Q_{y_0,l_0} \cap Q^q_j \neq \emptyset
\]
for some $0 \le j \le M_q$ with
$j \neq \sum_{k=1}^i M^q_{k-1} + i$ $(0 \le i \le M_{q-1})$.
Thus, for $q$ sufficiently large, by \eqref{u-nu} we have
\begin{equation}\label{p-meas-0}
\overline{Q}^q_j \subset Q_{y_0,2l_0} \subset U .
\end{equation}

 By Lemma~\ref{lem512}, for all $p>q$ and $1 \le k \le N$,
\[
\begin{aligned}
\bigl| Q^q_j \cap \{ Du_p \in \mathbb{P}(S_k^{r_{p+1}}(\lambda_{p+1})) \} \bigr|
&\ge \bigl| Q^q_j \cap \{ (Du_p, V_p) \in S_k^{r_{p+1}}(\lambda_{p+1}) \} \bigr| \\
&\ge \frac{1}{2}\,
\frac{\lambda_{q+2}}{\lambda_{p+1}}
(\lambda_{q+2}-\lambda_{q+1})
(\lambda_{q+2}-\delta_2)^{N-1}
\delta_1\, |Q^q_j| ,
\end{aligned}
\]
where $\mathbb{P}(\rho) = \rho^1$ for all
$\rho = (\rho^1,\rho^2) \in \R^{m \times n} \times \R^{m \times n}$.

Letting $p \to \infty$, and using the facts that
$u_p \to u$ strongly in $W^{1,1}(\Omega;\R^m)$,
$\lambda_p \to 1$, and $r_p \to r_0$, we obtain
\begin{equation}\label{p-meas}
\bigl| Q^q_j \cap \{ Du \in \overline{\mathbb{P}(S_k^{r_0}(1))} \} \bigr|
\ge \frac{1}{2}\,
\lambda_{q+2}
(\lambda_{q+2}-\lambda_{q+1})
(\lambda_{q+2}-\delta_2)^{N-1}
\delta_1\, |Q^q_j| > 0
\end{equation}
for all $1 \le k \le N$.

Let
\[
d_0 := \min_{k \neq l}
\dist\!\left(
\overline{\mathbb{P}(S_k^{r_0}(1))},
\overline{\mathbb{P}(S_l^{r_0}(1))}
\right).
\]
By condition (P1)(ii) of Condition $O_N$, we have $d_0 > 0$.
Hence, by \eqref{p-meas-0} and \eqref{p-meas},
\[
\|Du(y)-Du(z)\|_{L^\infty(U \times U)}
\ge \|Du(y)-Du(z)\|_{L^\infty(Q^q_j \times Q^q_j)}
\ge d_0 ,
\]
which yields
\[
\omega_{Du}(y_0) \ge d_0
\quad \text{for all } y_0 \in \Omega .
\]
Consequently, $Du$ is not essentially continuous at any point of $\Omega$. Hence, $u$ is nowhere $C^1$ in $\Omega.$

\medskip

This completes the proof of Theorem~\ref{mainthm}.

\subsection*{Acknowledgements}

M.~Liao was supported by the National Natural Science Foundation of China (No.~12401290) and the Natural Science Foundation of Jiangsu Province (No.~BK20230946).

The authors would like to thank Dr.~Seonghak Kim for his invaluable comments and suggestions, which significantly improved the presentation of this paper.


\begin{thebibliography}{99}

\bibitem{AF84}
E. Acerbi  and N. Fusco,  {\em Semicontinuity problems in the calculus of variations,}
Arch. Rational Mech. Anal., {\bf 86} (1984),  125--145.
 
       \bibitem{Ba77} J. M. Ball, {\em Convexity conditions and existence theorems in nonlinear elasticity,} Arch. Rational Mech. Anal., {\bf 63} (1977), 337--403.
 

\bibitem{BDM13}
{V. B\"ogelein, F. Duzaar and G. Mingione,} {\em The regularity of general parabolic systems with degenerate diffusion,} Memoirs of Amer. Math. Soc., {\bf 221} No. 1041, 2013.

 

  \bibitem{CZ92}
 J. Chabrowski and K. Zhang, {\em Quasi-monotonicity and perturbated systems with critical growth,} Indiana Univ. Math. J., {\bf 41}(2)  (1992), 483--504.


 
 \bibitem{CT25}   M. Colombo and R. Tione, {\em Non-classical solutions of the $p$-Laplace equation,} J. Eur. Math. Soc., {\bf 27}(12) (2025), 4845--4890. https://doi.org/10.4171/JEMS/1462
 

 \bibitem{Da08}
B. Dacorogna, ``Direct Methods in the Calculus of Variations," Second Edition. Springer-Verlag, Berlin, Heidelberg, New York, 2008.

\bibitem{DM97} B. Dacorogna and P. Marcellini, {\em General existence theorems for Hamilton-Jacobi equations in the scalar and vectorial cases,} Acta Math., {\bf 178}(1) (1997), 1--37.

\bibitem{DM99}
B. Dacorogna and P. Marcellini, ``Implicit Partial Differential Equations."  Birkh\"auser Boston, Inc., Boston, MA, 1999.



 \bibitem{DG00}
F.~Duzaar and J.~F.~Grotowski,
{\em Optimal interior partial regularity for nonlinear elliptic systems: The method of $A$-harmonic approximation,}
Manuscripta Math., \textbf{103} (2000), 267--298.


 

\bibitem{Ev86}
L.~C.~Evans, {\em Quasiconvexity and partial regularity in the
calculus of variations,} Arch. Rational Mech.
Anal., {\bf 95} (1986), 227--252.

\bibitem{Fu87}
 M. Fuchs, {\em Regularity theorems for nonlinear systems of partial differential equations
under natural ellipticity conditions,} Analysis, {\bf 7} (1987),  83--93.

\bibitem{FH85}
N. Fusco and J. Hutchinson, {\em$C^{1,\alpha}$ partial regularity of 
functions minimizing quasiconvex integrals,} Manuscripta Math., {\bf 54} (1985), 121--143.

\bibitem{Gi83} 
M. Giaquinta, ``Multiple Integrals in the Calculus of Variations and Nonlinear Elliptic
Systems."  Princeton Univ. Press, Princeton (1983).

 \bibitem{Gr86}
 M. Gromov,  ``Partial Differential Relations."   Springer-Verlag, Berlin, 1986.

\bibitem{GKY26}
B. Guo, S. Kim and B. Yan, {\em Irregular diffusions and loss of regularity in polyconvex gradient flows,} (Preprint) arXiv:2601.01035v1 [math.AP] 3 Jan 2026


\bibitem{Ha95}
C. Hamburger, {\em Quasimomotonicity, regularity and duality for nonlinear systems of partial differential equations,} Annali di Matematica pura ed applicata (IV), Vol. {\bf CLXIX} (1995),  321--354.

  

 \bibitem{Jo24}
 C. J. P. Johansson, {\em Wild solutions to  scalar Euler-Lagrange equations,} Trans. Amer. Math. Soc., {\bf 377}(7) (2024), 4931--4960.
  

\bibitem{KY15}
S. Kim and B. Yan, {\em Convex integration and infinitely many weak solutions to the Perona-Malik equation in all dimensions}, SIAM J. Math. Anal. {\bf 47}(4) (2015), 2770--2794.


\bibitem{KY17}
S. Kim and B. Yan, {\em On Lipschitz solutions for some  forward-backward parabolic  equations. II: the case against Fourier}, Calc. Var. {\bf 56}(3) (2017), Art. 67, 36 pp.

\bibitem{KY18}
S. Kim and B. Yan, {\em On Lipschitz solutions for some  forward-backward parabolic  equations}, Ann. Inst. H. Poincar\'e Anal. Non Lin\'eaire {\bf 35}(1) (2018), 65-100.


\bibitem{KY25}
S. Kim and B. Yan, {\em On integral convexity, variational solutions and nonlinear semigroups,} J. Math. Pures Appl. (9) {\bf 194} (2025), Paper No. 103662, 37 pp.

 \bibitem{La96}
R. Landes, {\em Quasimonotone versus pseudomonotone,} Proc. Royal Soc. Edinburgh, {\bf 126A} (1996), 705--717.

 \bibitem{Mi06}
G. Mingione, {\em Regularity of minima: an invitation to the dark side of the Calculus of Variations,} Application of Mathematics, {\bf 51} (2006), 355--426.

  \bibitem{Mo52}
C.~B.~Morrey,
{\em Quasiconvexity and the lower semicontinuity of multiple integrals,} Pacific J. Math., {\bf 2} (1952), 25--53.

 

\bibitem{MSv99}
 S. M\"uller and V. \v Sver\'ak, {\em Convex integration with constraints and applications to phase transitions and partial differential equations}, J. Eur. Math. Soc., {\bf 1} (4) (1999), 393--422.

\bibitem{MSv03} {S. M\"uller and V. \v Sver\'ak}, {\em Convex integration for Lipschitz mappings and counterexamples to regularity},  Ann. of Math. (2) {\bf 157}(3)   (2003), 715--742.

\bibitem{MSy01}
S.~M\"uller and M.~A.~Sychev, {\em Optimal existence theorems for
nonhomogeneous differential inclusions,} J. Funct. Anal., {\bf 181}(2)
(2001), 447--475.


 

  \bibitem{Sz04}
 L. Sz\'ekelyhidi, {\em The regularity of critical points of polyconvex functionals,} Arch. Rational Mech.  Anal.  {\bf 172}(1) (2004), 133--152.


 \bibitem{Ta93} L. Tartar, {\em Some remarks on separately convex functions,}  in ``Microstructure and
Phase Transitions," IMA Vol. Math. Appl., {\bf 54}  (D. Kinderlehrer, R. D. James,
M. Luskin and J. L. Ericksen, eds.), Springer-Verlag, New York (1993), 191--204.


 \bibitem{Ya20}
B.~Yan, {\em Convex integration for diffusion equations and Lipschitz
solutions of polyconvex gradient flows,} Calc. Var.  (2020), 59:123. https://doi.org/10.1007/s00526-020-01785-7

 \bibitem{Ya22}
B.~Yan, {\em On nonuniqueness and nonregularity for gradient flows of polyconvex functionals,}
Calc. Var. (2024) 63:4
https://doi.org/10.1007/s00526-023-02609-0

 
\bibitem{Zh86}
 K. Zhang, {\em On the Dirichlet problem for a class of quasilinear elliptic systems of
partial differential equations in divergence form,} in ``Proceedings  of Tianjin Conference on Partial Differential
Equations in 1986," (S. S. Chern ed.) Lecture Notes in Mathematics, {\bf 1306}, pp.
262--277. Springer, Berlin, Heidelberg, New York, 1988.


\bibitem{Zh92}
 K. Zhang, {\em Remarks on quasiconvexity and stability of equilibria for variational integrals,} 
 Proceedings of the American Mathematical Society, {\bf 114}(4) (1992),  927--930.
 
\bibitem{Zh06}
K. Zhang, {\em On existence of weak solutions for one-dimensional forward-backward diffusion equations}, J. Differential Equations {\bf 220} (2) (2006), 322--353.


\end{thebibliography}
\end{document}